\documentclass[reqno,11pt]{amsart}
\pdfoutput=1
\usepackage{amssymb,amsmath,amsthm}
\usepackage{setspace}
\usepackage{fullpage}
\usepackage{color}
\usepackage{graphicx,epsfig}
\usepackage{hyperref}

\newtheorem{theorem}  {Theorem}%[section]
\newtheorem{lemma}[theorem]  {Lemma} %[section]
\newtheorem{corollary}[theorem]  {Corollary}
\newtheorem{fact} [theorem] {Fact}
\newtheorem{claim} [theorem] {Claim}

\newtheorem{assumption}[theorem]  {Assumption}
\newtheorem{conj}{Conjecture}
\def\cA{{\mathcal A}}
\def\cB{{\mathcal B}}
\def\cF{{\mathcal F}}
\def\cG{{\mathcal G}}
\def\cH{{\mathcal H}}
\def\cW{{\mathcal W}}

\def\tF{\tilde{\cF}}  %  \CR
\def\hF{\dot{\cF}}  %  \ST
\def\dhF{\ddot{\cF}}  %  \DT
\def\bF{\bar{\cF}}  
\def\tA{\tilde{\cA}}  %  \CR
\def\hA{\dot{\cA}}  %  \ST
\def\dhA{\ddot{\cA}}  %  \DT
\def\tB{\tilde{\cB}}  %  \CR
\def\hB{\dot{\cB}}  %  \ST
\def\dhB{\ddot{\cB}}  %  \DT
\def\mup{\mu_p}
\def\muA{\mu_p(\cA)}

\def\({\left(}
\def\){\right)}
\def\dual{{\rm dual}}

\def\shiftsto{\to}

\newcommand{\claimproof}{\renewcommand{\qedsymbol}{$\diamond$}}
\begin{document}
\title{Towards extending the Ahlswede--Khachatrian theorem to cross $t$-intersecting families}

\author[sjl]{Sang June Lee}
\address{Department of Mathematics, Duksung Women's University, Seoul 132-714, South Korea}
\email{sanglee242@duksung.ac.kr}
\thanks{The first author was supported by Basic Science Research Program through the National Research Foundation of Korea (NRF) funded by the Ministry of Science, ICT \& Future Planning (NRF-2013R1A1A1059913)}

\author{Mark Siggers}
\address{College of Natural Sciences, Kyungpook National University,
          Daegu 702-701, South Korea}
\email{mhsiggers@knu.ac.kr}
\thanks{The second author was supported by the Korean NRF Basic Science Research Program (2014-06060000) funded by the Korean government (MEST)}

\author{Norihide Tokushige}
\address{College of Education, Ryukyu University, Nishihara, Okinawa 903-0213, Japan}
\email{hide@edu.u-ryukyu.ac.jp}
\thanks{The last author was supported by JSPS KAKENHI 25287031}


\begin{abstract}
  Ahlswede and Khachatrian's diametric theorem is a weighted version of their
  complete intersection theorem, which is itself a well known extension 
  of the $t$-intersecting Erd\H os-Ko-Rado theorem. The complete intersection theorem
  says that the maximum size of a family of subsets of $[n] = \{1, \dots, n\}$,  every pair
  of which intersects in at least $t$ elements, is the size of certain trivially 
  intersecting families proposed by Frankl.  
   We address a cross intersecting version of their diametric theorem.  

  Two families $\cA$ and $\cB$ of subsets of $[n]$ are
  {\em cross $t$-intersecting} if for every $A \in \cA$ and $B \in \cB$,
  $A$ and $B$ intersect in at least $t$ elements. 
  The $p$-weight of a $k$ element subset $A$ of $[n]$ is $p^{k}(1-p)^{n-k}$, and the
  weight of a family $\cA$  is the sum of the weights of its sets.  
  The weight of a pair of families is the product of the weights of the families. 

  The  maximum $p$-weight of a $t$-intersecting family depends on the value of $p$. 
  Ahlswede and Khachatrian showed that for $p$ in the range
  $[\frac{r}{t + 2r -  1}, \frac{r+1}{t + 2r + 1}]$, the 
  maximum $p$-weight of a $t$-intersecting family is that of the family $\cF^t_r$ 
  consisting of all subsets of $[n]$ containing at least $t+r$ elements of the set
  $[t+2r]$.

  In a previous paper we showed a cross $t$-intersecting version of this for large $t$ in
  the case that $r = 0$. In this paper, we do the same in the case that $r = 1$. 
  We show that for $p$ in the range $[\frac{1}{t + 1}, \frac{2}{t + 3}]$
  the maximum $p$-weight of a cross $t$-intersecting pair of
  families, for $t \geq 200$, is achieved when both families are $\cF^t_1$. Further, we show that
  except at the endpoints of this range, this is, up to isomorphism, the only pair of
  $t$-intersecting  families achieving this weight.
\end{abstract}

\maketitle

\section{Introduction}
Let $[n]:=\{1,2,\dots,n\}$ and let $\binom{[n]}{k}$ be the family of all $k$-subsets of $[n]$. For a positive integer $t$, 
the family $\cA\subset 2^{[n]}$ is called \emph{$t$-intersecting} if, for each 
$A,A'\in \cA$, we have $|A\cap A'|\geq t$. 
Erd\H os, Ko, and Rado proved in \cite{EKR} that, 
for each $k$ and $t$, there exists $n_0=n_0(k,t)$ such that if $n\geq n_0$ and
a family of $k$-element subsets $\cA\subset\binom{[n]}k$ is $t$-intersecting,
then $|\cA|\leq\binom{n-t}{k-t}$ with equality holding
if and only if there is some $T\in\binom{[n]}t$ such that
$\cA=\{A\in\binom{{n}}k:T\subset A\}$. 
The exact bound $n_0(k,t)=(t+1)(k-t+1)$ was established by
Frankl \cite{Fckt}, where he introduced the random walk method,  %using the random walk method that he invented, 
and independently by Wilson \cite{W}, where he used a linear programming bound due to
Delsarte. 

Frankl also considered the case when $n<(t+1)(k-t+1)$. He defined
$t$-intersecting families $\cF_i^t$ by
\[
\cF_i^t=\cF_i^t(n)=\Big\{F\subset[n]:\big|F\cap[t+2i]\big|\geq t+i\Big\},  
\]
and conjectured that if $\cA\subset\binom{[n]}k$ is $t$-intersecting, 
then 
\[
|\cA|\leq\max_i\big|\cF_i^t\cap\tbinom{[n]}k\big|.  
\]
This conjecture was partially proved by Frankl and F\"uredi in \cite{FF},
and was finally settled by Ahlswede and Khachatrian in the affirmative  
in \cite{AK1} and \cite{AK2}.
This result, now known as the complete intersection theorem, 
is one of the most important results in extremal set theory.

Ahlswede and Khachatrian also obtained the $p$-weight version of their
complete intersection theorem in \cite{AK-p}. This result, which they called the 
diametric theorem, applies to non-uniform families of subsets of $[n]$.
To state the result, we let $p$ be a real number with $0<p<1$, and let $q:=1-p$.
For a family $\cF\subset 2^{[n]}$, the \emph{$p$-weight of $\cF$} is defined by
\[ \mup(\cF) := \sum_{F \in \cF}p^{|F|}q^{n-|F|}. \]
Ahlswede and Khachatrian showed that for $p \leq 1/2$ if $\cF\subset 2^{[n]}$ is 
$t$-intersecting, then 
\begin{equation}\label{AK-result}
\mu_p(\cF)\leq\max_i\mu_p(\cF_i^t). 
\end{equation}
Comparing $\mu_p(\cF_i^t)$ and $\mu_p(\cF_{i+1}^t)$, it can be shown that % it is readily seen that
$\max_i\mu_p(\cF_i^t)=\mu_p(\cF_r^t)$ if and only if
\begin{equation}\label{p-range}
 \frac{r}{t+2r-1}\leq p\leq\frac{r+1}{t+2r+1}. 
\end{equation}
All values of $p \in (0,1/2)$ fall into this range for some $r$, larger $p$ yield larger $r$.

For a positive integer $t$, the families $\cA, \cB\subset 2^{[n]}$ are called \emph{cross $t$-intersecting} if, for each $A\in \cA$ and $B\in \cB$, we have $|A\cap B|\geq t$. 
We consider an extension of \eqref{AK-result} to cross $t$-intersecting 
families.
%\begin{conj}
%If $\cA\subset 2^{[n]}$ and $\cB\subset 2^{[n]}$ are cross $t$-intersecting,
%then 
%\[
%\mu_p(\cA)\mu_p(\cB)\leq\max_i\left(\mu_p(\cF_i^t)\right)^2. 
%\]
%\end{conj}
\begin{conj}\label{conj:main}
If $\cA\subset 2^{[n]}$ and $\cB\subset 2^{[n]}$ are cross $t$-intersecting,
then where $r$ is such that $p$ satisfies \eqref{p-range}, 
\[
\mu_p(\cA)\mu_p(\cB)\leq \mu_p(\cF_r^t)^2. 
\]
\end{conj}
With Frankl, in \cite{F}, we 
verified the $r = 0$ case of the above conjecture for $t \geq 14$. 
%if $t\geq 14$.
%considered the case when 
%$p$ satisfies \eqref{p-range} for $r=0$, and verified the above conjecture
%if $t\geq 14$.
 In this paper we %continue the study in this direction.
%Namely we consider the case when $p$ satisfies \eqref{p-range} for $r=1$, and
%we verify the conjecture if $t\geq 200$.
verify the $r=1$ case of the conjecture for $t \geq 200$. 
 This result is perhaps the first
result concerning cross intersecting families, where optimal structures are
different from the so-called trivial structure $\cF_0^t$.
To state our main result we need one more definition.
Two families $\cG_1,\cG_2\in 2^{[n]}$ are \emph{isomorphic}, denoted by $\cG_1\cong \cG_2$,
if there is a permutation $\sigma$ on $[n]$ such that $\cG_1=\Big\{\{\sigma(k): k\in G\}\;:\; G\in \cG_2\Big\}$.

\begin{theorem}\label{thm:main}
Let $n$ and $t$ be integers with $n\geq t\geq 200$, and let $p$ be such that 
 $\frac 1{t+1}\leq p\leq \frac 2{t+3}$.
If $\cA\subset 2^{[n]}$ and $\cB\subset 2^{[n]}$ are 
cross $t$-intersecting, then
\begin{equation}\label{eq:main}
 \mup(\cA)\mup(\cB)\leq \(\mup(\cF^t_1)\)^2=\Big( (t+2)p^{t+1}q + p^{t+2} \Big)^2.
\end{equation}
Moreover, equality holds if and only if one of the following holds:

%\begin{align}\label{eq:equality}
%(i) \hskip 1em & \cA=\cB\cong \cF^t_0 \hskip 1em \mbox{and} \hskip 4.5em p=\frac{1}{t+1}, \nonumber\\
%(ii)  \hskip 1em &\cA=\cB\cong \cF^t_1\hskip 1em \mbox{and} \hskip 1em \frac{1}{t+1}\leq p\leq \frac{2}{t+3}, \\
%(iii)   \hskip 1em &\cA=\cB\cong \cF^t_2 \hskip 1em \mbox{and}\hskip 4.5em  p=\frac{2}{t+3}.\nonumber
%\end{align}
\begin{enumerate}
\item $\cA=\cB\cong \cF^t_0$ and $p=\dfrac{1}{t+1}$,
\item $\cA=\cB\cong \cF^t_1$ and $\dfrac{1}{t+1}\leq p\leq \dfrac{2}{t+3}$,
\item $\cA=\cB\cong \cF^t_2$ and $p=\dfrac{2}{t+3}$.
\end{enumerate}
\end{theorem}

Remark that we do not attempt to optimize the range of $t$. The parts requiring $t$ to be around $200$
are~\eqref{eq:h(u,s,p)} and~\eqref{eq:t=180}.

\smallskip\noindent\textbf{Organization:} 
 In Section \ref{sec:prelims} we introduce some standard definitions and techniques,
 and state some useful results from \cite{F}.
 % We also prove one small extension of one of these results. 
 In Section \ref{sec:setup} we make some quick 
 reductions and setup parameters for the families $\cA$ and $\cB$ by which we
 break the proof down into cases.

 In particular, we
 introduce a pair of parameters $(s,s')$, with  $0 \leq s' \leq s$,  which 
 effectively 
 measures the difference  between our families $\cA$ and $\cB$ and the optimal
 families $\cF^t_0$, $\cF^t_1$ or $\cF^t_2$.
  When $(s,s')$ is one of $(0,0), (1,1), (2,2), (1,0),$ or $(2,1)$,
  then $\cA$ and $\cB$ will be, or
  will be very close to, one of these families.  In this case we have
  to look closely at the structure of our families, and compare them with
  the optimal families directly. This will be done in Section  \ref{sec:non diagonal cases}.

  The remaining cases are dealt with in Section \ref{sec:easy cases}. 
  When $(s,s')$ is not one of the above five values, then $\cA$ and $\cB$
   are very different from the optimal families, so we can expect them to have
   relatively small weight. This
  seems as though it should make computation easier, but there is an added
  difficulty in that we can no longer compute their weight relative to the
  optimal families, rather we must compute these weights directly. That said,
  if $s$ is big, a fairly crude estimation of the weight will suffice, and
  these cases are done in Subsection~\ref{sub:large_s}.
  For the intermediate values of $s$ we consider a finer bound on the size 
  of the families, and use its monotonicity on the range 
  $2 \leq s' \leq s \leq 10$ to show achieve our bound for most of these 
  values.  This finer bound is still too crude for the final five cases. 

  This overall approach is based on the paper \cite{F}, but the monotonicity
  ideas used in Subsection \ref{sub:medium_s} are new. 
 % We feel that such ideas will be necessary in proving
 % the Theorem \ref{thm:main} for smaller values of $p$, (i.e. for large values of $r$ in 
 % \eqref{p-range}).  
  We feel that such ideas will be necessary in proving Conjecture \ref{conj:main}
  for larger values of $r$.    
See \cite{PT} for some recent developments on cross-intersecting families in
different directions.

%%%%%%%%%%%%%%

\section{Preliminaries}\label{sec:prelims}

\subsection{Subset vs. walk on a two-dimensional grid}

It is useful to regard a set $F\subset [n]$ as a walk starting at the origin 
$(0,0)$ of the two-dimensional grid $\mathbb{Z}^2$ as follows.
If $i\in F$, then the $i$-th step is \emph{up} from $(x,y)$ to $(x,y+1)$.
Otherwise, the $i$-th step is \emph{right} from $(x,y)$ to $(x+1,y)$.
For simplicity, we refer to $F\subset [n]$ as a set or a walk. 
See Figure~\ref{fig:walk} for an example.
\begin{figure}
\begin{center}
\includegraphics[scale=0.23]{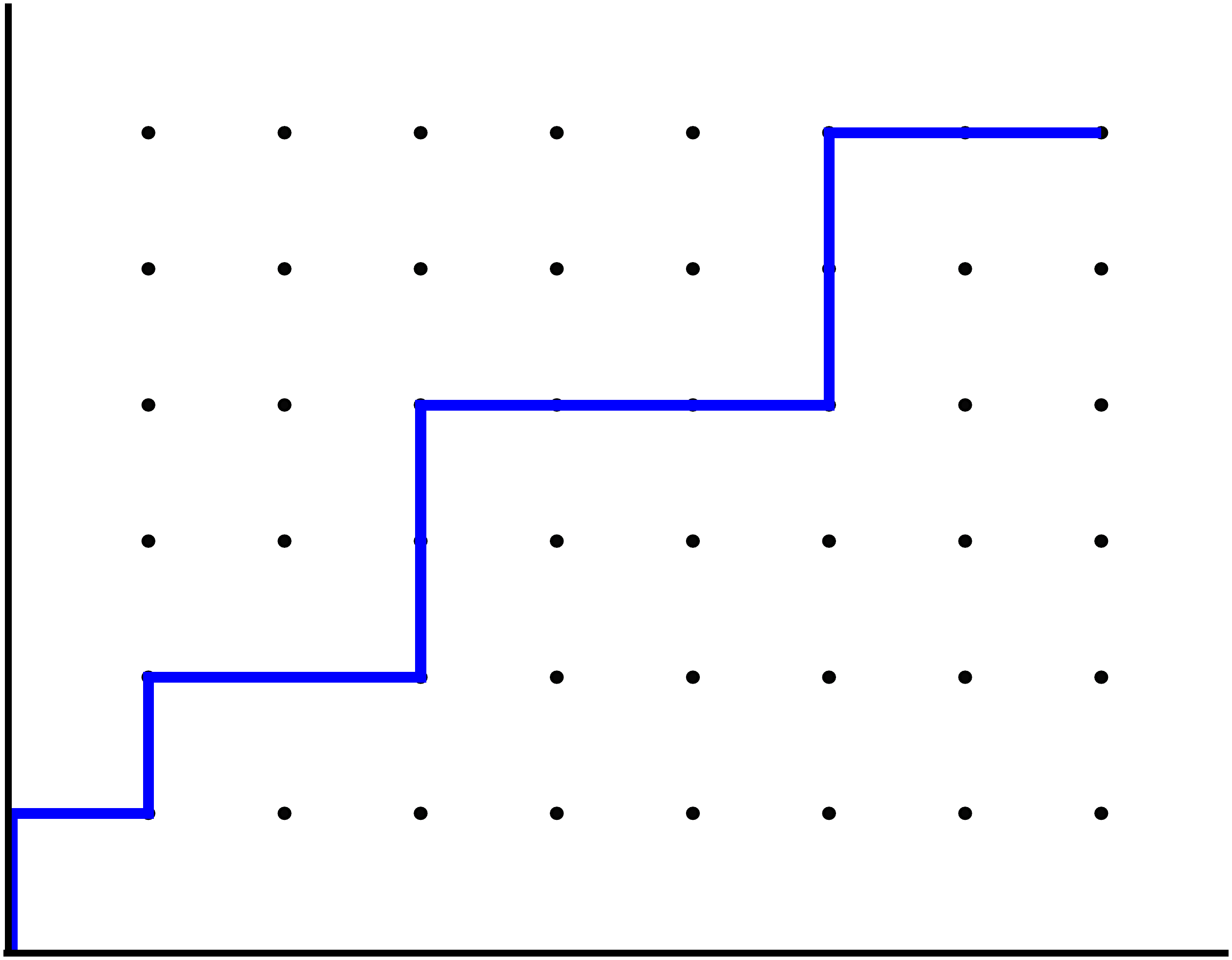}
\end{center}
\caption{The walk $F=\{1, 3, 6, 7, 11, 12\}\subset [14]$}
\label{fig:walk}
\end{figure}

Let $\cF^\ell$ be the family of all walks
that hit the line $y = x + \ell$; that is, let
\[
\cF^\ell=\Big\{F\subset[n]: \, \big|F\cap[j]\big|\geq \frac{j+\ell}{2} \text{ for some }j\Big\}. 
\]
Partition the family $\cF^\ell$ into the following three subfamilies:
\begin{align*}
\tF^\ell:=&\big\{F\in\cF^\ell: \, F \text{ hits $y = x + \ell+1$}\big\},\\
\hF^\ell:=&\big\{F\in\cF^\ell: \, F \text{ hits $y = x + \ell$ exactly once, 
but does not hit $y=x+\ell+1$}\big\}, \\
\dhF^\ell:=
& \big\{F\in\cF^\ell: \, F \text{ hits $y = x + \ell$ at least twice, 
but does not hit $y=x+\ell+1$}\big\}. 
\end{align*}
So we can write 
\begin{equation*}%\label{eq:partition of F^u}
   \cF^\ell=\tF^\ell\sqcup \hF^\ell \sqcup \dhF^\ell.
\end{equation*}

The following lemmas hold.
\begin{lemma}[\cite{F}, Lemma 2.2 (i,iii)]\label{lem:mu_F} For any positive integer $\ell$, we have the following, where $\alpha=p/q$. 
\begin{enumerate}
\item $\mup(\cF^\ell)\leq \alpha^\ell$ and $\mup(\tF^\ell)\leq \alpha^{\ell+1}$
\item $\mup(\dhF^\ell)\leq \alpha^{\ell+1}$
\end{enumerate}
\end{lemma}

\begin{lemma}[\cite{F}, Lemma 2.2 (ii)]\label{lem:epsilon}
For every $\epsilon>0$, there exists an $n_0$ such that if 
$n$ and $l$ are integers satisfying $n\geq n_0$ and $l\geq 1$, 
then the following holds: If $\cF\subset 2^{[n]}$
and no walk in $\cF$ hits the line $y=x+\ell$, then
$$\mup(\cF)<1-\alpha^\ell+\epsilon.$$
\end{lemma}

\subsection{Inclusion maximal and shifted families}\label{sec:shifted}
A family $\cF\subset 2^{[n]}$ is called \emph{inclusion maximal} if $F\in \cF$ and $F\subset F'$ imply $F'\in \cF$.
\begin{fact}\label{fact:inclusion_maximal} If $\cA, \cB$ are cross $t$-intersecting families in $2^{[n]}$, then there are inclusion maximal cross $t$-intersecting
families $\cA', \cB'\in 2^{[n]}$ such that $\cA\subset \cA'$ and $\cB\subset \cB'$.
\end{fact}

For $F\subset [n]$ and $i,j\in [n]$, let
\[ 
s_{ij}(F) := \begin{cases} 
      (F \setminus \{j\})\cup \{i\}  & \text{if } F\cap\{i,j\}=\{j\} 
       \text{ and } (F \setminus \{j\})\cup \{i\}\not\in\cF,\\
    F & \text{otherwise.} \\
              \end{cases}  
\]
Then, for $\cF\subset 2^{[n]}$, let $$s_{ij}(\cF):=\{s_{ij}(F):F\in\cF\}.$$ 
A family $\cF$ is called \emph{shifted} if $s_{ij}(\cF) = \cF$ 
for all $1\leq i < j \leq n$. 
Here we list some basic properties concerning shifting operations.

\begin{lemma}[\cite{F}, Lemma 2.3]\label{lem:shifting}
Let $1\leq i<j\leq n$ and let $\cF,\cG\subset 2^{[n]}$.
\begin{enumerate}
\item Shifting operations preserve the $p$-weight of a family, that is,
$\mu_p(s_{ij}(\cG))=\mu_p(\cG)$.
\item If $\cF$ and $\cG$ in $2^{[n]}$ are cross $t$-intersecting families, 
then $s_{ij}(\cF)$ and $s_{ij}(\cG)$ are cross $t$-intersecting 
families as well.
\item For a pair of families we can always obtain a pair of 
shifted families by repeatedly shifting families simultaneously 
finitely many times.
%\item If $\cG$ is inclusion maximal, and $s_{ij}(\cG)=\cF_\ell^t(n)$, 
%then $\cG\cong\cF_\ell^t(n)$ for $\ell=0,1$.
\end{enumerate}
\end{lemma}

The following lemma,  which mimics a proposition in \cite{AK1}, that  
is in turn based on an idea from \cite{Fckt}, was stated in \cite{F}, but
its proof was only sketched.  We prove it now for all values of $r$, though
we only need it for $r = 0,1,$ and $2$. 

\begin{lemma}\label{lem:shifting2}
Let $t\geq 2$ and let $\cA,\cB\subset 2^{[n]}$ be 
cross $t$-intersecting families. 
If $s_{ij}(\cA)=s_{ij}(\cB)=\cF_r^t$, then $\cA=\cB\cong\cF_r^t$.
\end{lemma}
\begin{proof}
 First we remark that if $\cA\cong\cF_r^t$, then $\cB=\cA$.
%(So if $\cA\cong\cF_r^t$, then $\cB=\cA\cong\cF_r^t$.)
Further, if $i,j\in[t+2r]$ or $i,j\not\in[t+2r]$, then $\cA=\cF_r^t$ and we are done.
So without loss of generality we may assume that $i=t+2r$ and $j=n$.
Define two subfamilies $\cA_1$ and $\cA_2$ of $\cA$ by
\begin{align*}
\cA_1&=\{A\in\cA:|A|=t+r,\,i\not\in A,\, j\in A,\,(A\cup\{i\})\setminus\{j\}\not\in\cA\},\\
\cA_2&=\{A\in\cA:|A|=t+r,\,i\in A,\, j\not\in A,\,(A\setminus\{i\})\cup\{j\}\not\in\cA\}.
\end{align*}
Since $s_{ij}(\cA)=\cF_r^t$, we have 
$|A\cap[t+2r-1]|=t+r-1$ for all $A\in\cA_1\cup\cA_2$.
If $\cA_1=\emptyset$, then $\cA=\cF_r^t$.
If $\cA_2=\emptyset$, then 
$\cA=\{A\subset[n]:|A\cap([i-1]\cup\{j\})|\geq t+r\}\cong\cF_r^t$. 
So we may assume that $\cA_1\neq\emptyset$ and $\cA_2\neq\emptyset$.

Let $\cH=\binom{[t+2r-1]}{t+r-1}$.
Then for every $H\in\cH$ we have $H\cup\{j\}\in\cA_1$ or $H\cup\{i\}\in\cA_2$
(but not both). Thus we can identify $\cH$ with $\cA_1\sqcup\cA_2$.
We also define $\cB_1$ and $\cB_2$ in the same manner. 
Let $\cH'$ be a copy of $\cH$, and identify $\cH'$ with $\cB_1\sqcup\cB_2$.

Now we define a bipartite graph $G$ on $V(G)=\cH\sqcup\cH'$, by letting
$\{H,H'\}$ be an edge, for $H\in\cH$ and  $H'\in\cH'$,  if $|H\cap H'|=t-1$.  
We claim that $G$ is a connected graph. Indeed, the graph $G_0$ defined 
on $\cH$ by letting $\{H_1,H_2\}$ be an edge if $|H_1\cap H_2|=t-1$, is Kneser's
graph; and this is connected and non-bipartite for $t>1$.
If $G$ is not connected, then each connected
component is isomorphic to $G_0$, which contradicts the fact that $G_0$
is not bipartite. This shows that $G$ is connected. 

Therefore, there is a path from $A\in\cA_1$ to $B\in\cB_2$
in $G$, and on this path there is 
an edge $\{A_1,B_2\}$ where $A_1\in\cA_1$,  $B_2\in\cB_2$, or 
an edge $\{A_2,B_1\}$ where $A_2\in\cA_2$,  $B_1\in\cB_1$.
But then $|A_1\cap B_2|=t-1$ or $|A_2\cap B_1|=t-1$, which contradicts
the fact that $\cA$ and $\cB$ are cross $t$-intersecting. 
\end{proof}

 Fact~\ref{fact:inclusion_maximal} and Lemmas~\ref{lem:shifting} allow us to assume that
 $\cA$ and $\cB$ are inclusion maximal and shifted in proving the inequalities in 
 Theorem \ref{thm:main}. Lemma~\ref{lem:shifting2}, allows us to extend this assumption
 to the uniqueness results in the case of equality in the Theorem. We record this as
 the following assumption.

\begin{assumption}\label{assum:maximal_shifted}
$\cA$ and $\cB$ are inclusion maximal and shifted.
\end{assumption}
%%%%%%%%%%%%%%

\section{Setup for proof of Theorem~\ref{thm:main}}\label{sec:setup}

Recall that $n$ and $t$ are integers with $n\geq t\geq 200$, and $p$ is a real number with
$\frac{1}{t+1}\leq p\leq\frac {2}{t+3}$. Set $q=1-p$ and $\alpha=p/q$.
The following holds.
\begin{lemma}[\cite{F}, Lemma 2.12] Let $f(n)$ be the maximum of $\mup(\cA)\mup(\cB)$ 
over all pairs $\cA$ and $\cB$ of cross $t$-intersecting families in $2^{[n]}$. Then,
$f(n)\leq f(n+1).$
\end{lemma}
We may therefore assume that $n$ is arbitrarily large.

% Recalling~Assumption~\ref{assum:maximal_shifted}, we assume that $\cA$ and $\cB$ are inclusion maximal, shifted,
%cross $t$-intersecting families in $2^{[n]}$. (See Section~\ref{sec:shifted} for the definition of an inclusion maximal, shifted family.)

For $\cF\subset 2^{[n]}$, let $\lambda(\cF)$ be the maximum integer $\lambda \geq 0$ 
such that all walks in $\cF$ hit the line $y = x + \lambda$.
Let $u=\lambda(\cA)$ and $v=\lambda(\cB)$. The following holds.
\begin{lemma}[\cite{F}, Lemma 2.11(ii)]
If $\cA$ and $\cB$ are shifted, inclusion maximal, cross $t$-intersecting families in $2^{[n]}$, then
$\lambda(A)+\lambda(B)\geq 2t.$
\end{lemma}
Therefore, we assume that $u+v\geq 2t$. If $u + v \geq 2t+1$, then Lemma~\ref{lem:mu_F} gives that
 \[ \mup(\cA)\mup(\cB)\leq \mup(\cF^u)\mup(\cF^v) \leq 
\alpha^u\alpha^v\leq \alpha^{2t+1}.\]

One can check that $\alpha^{2t+1}< 0.99 \big(\mup(\cF_1^t)\big)^2$ for $t \geq 26$. Indeed, 
\begin{align*}
\frac{\big(\mup(\cF_1^t)\big)^2}{\alpha^{2t+1}} &\geq (t+2)^2pq^3(q^t)^2\geq (t+2)^2pq^3\frac{1}{e^4},
\end{align*}
where the second inequality follows from
\begin{equation}\label{eq:q^t}
\frac{1}{e^2} < \(1 - \frac{2}{t+3}\)^{t+3} < \(\frac{t+1}{t+3}\)^t \leq q^t\leq \(\frac{t}{t+1}\)^t \leq 0.5.
\end{equation}
Since $pq^3$ is increasing in $p$ for $p\leq 0.25$, we have that
\begin{equation*}
\frac{\big(\mup(\cF_1^t)\big)^2}{\alpha^{2t+1}} \geq \frac{(t+2)^2 t^3}{e^4(t+1)^4}>1.02,
\end{equation*}
where the last inequality holds for $t\geq 26$. 

Therefore, we assume that $$u+v=2t.$$ Without loss of generality, let $$u\leq v.$$

Note that $\cA \subset \cF^u$. So $\cA$ is partitioned as $\cA=\tA\sqcup \hA \sqcup \dhA$, where
\[
 \,  
\tA := \cA \cap \tF^u, \, \hA:=\cA \cap \hF^u,  \,\text{ and } \,  \dhA := \cA \cap \dhF^u.   
\]
Similarly, we have that $\cB=\tB\sqcup \hB \sqcup \dhB$, where
\[
\tB := \cB \cap \tF^v, \, \hB:=\cB \cap \hF^v,  \,\text{ and } \,  \dhB := \cB \cap \dhF^v.   
\]

If $\hA=\emptyset$, then $\cA= \dhA\cup \tA$, and hence, $$\muA=\mup(\tA)+\mup(\dhA)\leq \mup(\tF^u)+\mup(\dhF^u)\leq \alpha^{u+1}+\alpha^{u+1},$$
where the last inequality follows from Lemma~\ref{lem:mu_F}.
Thus, we have that
\begin{equation*}%\label{hA=empty}
\mup(\cA)\mup(\cB)\leq(\alpha^{u+1}+\alpha^{u+1})\alpha^v
\leq 2\alpha^{2t+1}< 0.99\(\mup(\cF_1^t)\)^2,
\end{equation*}
where the last inequality holds for $t\geq 110$.
Similarly, we have that if $\hB=\emptyset$, then, for $t\geq 110$, $$\mup(\cA)\mup(\cB)<0.99 \(\mup(\cF_1^t)\)^2.$$  
So~\eqref{eq:main} holds if $\hA=\emptyset$ or $\hB=\emptyset$.

We may therefore assume that $\hA\neq\emptyset$ and $\hB\neq\emptyset$. Recall that $$\cF^\ell_i=\Big\{F\subset[n]:\big|F\cap[\ell+2i]\big|\geq \ell+i\Big\}.$$
That is, $\cF^\ell_i$ is the family of walks hitting $(i, i+k)$ for some $k\geq \ell$.
Note that as $\hA$ and $\hB$ are non-empty, there exist non-negative integers $s$ and $s'$ such that $\hA\cap\cF_s^u\neq\emptyset$ and 
$\hB\cap\cF_{s'}^v\neq\emptyset$. 
The next lemma tells us that such $s$ and $s'$  are unique. Its statement has been modified, 
but it is essentially Lemma 3.2 of~\cite{F}.

\begin{lemma}[\cite{F}, Lemma 3.2]\label{lem:structure} Suppose that $\hA\neq\emptyset$ and $\hB\neq\emptyset$. Then,
there exist unique non-negative integers $s$ and $s'$ such that
 \begin{equation*}%\label{eq:cA_s}
 \cA_s:=\hA\sqcup\dhA\subset\cF_s^u \hskip 1em
\mbox{and} \hskip 1em \cB_{s'}:=\hB\sqcup\dhB\subset\cF_{s'}^v.\end{equation*} Moreover, $s-s'=(v-u)/2.$
In particular, $s\geq s'$.
\end{lemma}

Here, we record our {\bf setup}. 
\begin{itemize}
\item $\cA$ and $\cB$ are shifted maximal cross $t$-intersecting families. 
\item $n$ may be assumed to be arbitrarily large.  %$n\geq n_0$ with $n_0=n_0(t)$.
\item $q = 1-p$ and $\alpha=p/q$.
\item $t\geq 200$, and $\frac{1}{t+1}\leq p\leq\frac {2}{t+3}$, so $\frac{t+1}{t+3}\leq q\leq\frac {t}{t+1}$.
\item $u+v=2t$ and $1\leq u\leq t\leq v\leq 2t$.
\item $s\geq s' \geq 0$ and
$s-s'=(v-u)/2$. 
\item  $u=t-s+s'$ and $v=t+s-s'$.
\item $\cA=\hA\sqcup\dhA\sqcup\tA\subset\cF^u$ and 
$\cB=\hB\sqcup\dhB\sqcup\tB\subset\cF^v$.
\item 
$\hA\neq\emptyset$, $\hB\neq\emptyset$,  $\hA\sqcup\dhA\subset\cF_s^u$, 
and $\hB\sqcup\dhB\subset\cF_{s'}^v$.
\end{itemize}

\section{Almost all cases}\label{sec:easy cases}

The rest of the proof is broken down into cases based on the value of $(s,s')$. 
We first deal with the cases with  $s \geq 10$. Then we  spend the rest of the
section  reducing  the remaining cases to the five final cases which will be proved in
Section \ref{sec:non diagonal cases}.

\subsection{ Large values of $s$}\label{sub:large_s}

Let 
\begin{equation}\label{def:bF}
\bF^\ell_i:=(\hF^\ell\cup\dhF^\ell)\cap\cF^\ell_i.
\end{equation}
We use the following key estimation from \cite{F}.
\begin{claim}[\cite{F}, Claim 3.3]\label{clm:estimation}
There is an integer $n_0$ such that if $n>n_0$, then
\[
 \mup(\tF^\ell\cup\bF^\ell_i)<f(\ell,i,p)\cdot 1.001,
\]
where 
\begin{equation}\label{eq:f(ell,i,p)}
 f(\ell,i,p):=\alpha^{\ell+1}+\binom{\ell+2i}i\frac{\ell+1}{\ell+i+1}p^{\ell+i}q^i(1-\alpha).
\end{equation}
\end{claim}

Now, since
$\cA\subset\tF^u\cup\bF^u_s$ and $\cB\subset\tF^v\cup\bF^v_{s'}$, we have that
$$\mup(\cA)\mup(\cB)\leq \mup(\tF^u\cup\bF^u_s)\mup(\tF^v\cup\bF^v_{s'})\leq f(u,s,p)f(v,s',p)\cdot 1.001^2.$$
Hence, in order to show~\eqref{eq:main}, it suffices to show that
\begin{equation*}%\label{eq:g(s,s')}
g(s,s'):= f(u,s,p)f(v,s',p)<0.99\(\mup(\cF^t_1)\)^2. 
\end{equation*}
Observe that $g(s,s')$ depends on $t,p,s$ and $s'$, but for simplicity we only write
 the variables $s$ and $s'$.

\begin{claim}\label{clm:s_large} For $s\geq 10$, we have $g(s,s')<0.99\(\mup(\cF^t_1)\)^2.$
\end{claim}

\begin{proof}
Set $h(\ell,i,p):= p^{-\ell}f(\ell,i,p),$ so that 
\begin{equation*}%\label{eq:h}
 h(\ell,i,p)=
\frac{p}{q^{\ell+1}}+\binom{\ell+2i}i\frac{\ell+1}{\ell+i+1}(pq)^i(1-\alpha).
\end{equation*}
 Since $p^{2t}\leq \(\mup(\cF^t_1)\)^2$,  it suffices to show that 
 \begin{equation}\label{eq:h_u,v}h(u,s,p)h(v,s',p)< 0.99.\end{equation}

First, we estimate $h(u,s,p)$.
We have that
\begin{equation*}
h(u,s,p)\leq h(t,s,p)\leq h\Big(t,s, \frac{2}{t+3}\Big),
\end{equation*}
where the second inequality holds since $p$, $1/q$, $pq$, and $pq(1-\alpha)$ are increasing in $p$ for $0<p\leq 2/(t+3)\leq 0.25$.
Consequently, since $p/q^{t+1}=\frac{2}{t+3}(\frac{t+3}{t+1})^{t+1}\leq 0.073$ for $t\geq 200$, we have
\begin{align*}
h\Big(t,s, \frac{2}{t+3}\Big)&\leq  0.073+\binom{t+2s}{s}\Big(\frac{2}{t+3}\Big)^s 
\leq 0.073 + \Big(\frac{e(t+2s)}{s}\frac{2}{t+3}\Big)^s\\
&= 0.073 + \Big(\frac{2e(t+2s)}{s(t+3)}\Big)^s.
\end{align*}
Note that
$\frac{2e(t+2s)}{s(t+3)}<1$ if and only if
$s>\frac{2et}{t+3-4e}. $
Since $s\geq 6>\frac{2et}{t+3-4e}$ for $t\geq 84$, we have that $\frac{2e(t+2s)}{s(t+3)}<1$. Also, $\phi(s):=\frac{2e(t+2s)}{s(t+3)}$ is strictly decreasing in $s$, since
$\frac{\phi(s+1)}{\phi(s)}=\frac{(t+2s+2)s}{(t+2s)(s+1)}<1.$
Therefore, for $s\geq 10$,
\begin{equation}\label{eq:h(u,s,p)}
h\Big(t,s, \frac{2}{t+3}\Big)\leq  0.073 + \Big(\frac{2e(t+20)}{10(t+3)}\Big)^{10}<0.08.
\end{equation}

Next, we estimate $h(v,s',p)$. Similar to the estimation of $h(u,s,p)$, we have that
\begin{equation*}
h(v,s',p)\leq h(2t,s',p)\leq h\Big(2t,s', \frac{2}{t+3}\Big).
\end{equation*}
Consequently, since $p/q^{2t+1}\leq 0.53$ for $t\geq 200$, we infer that
\begin{align*}
h\Big(2t,s', \frac{2}{t+3}\Big)&\leq  0.53+\binom{2t+2s'}{s'}\Big(\frac{2}{t+3}\Big)^{s'} 
\leq  0.53+ \psi_a\psi_b,
%\leq  0.53+ \Big(\frac{4e(t+s')}{s'(t+3)}\Big)^{s'}.
\end{align*}
% Note that
% $\frac{4e(t+s')}{s'(t+3)}<1$ if and only if
% ${s'}>\frac{4et}{t+3-4e}.$ 
% If $s'\geq 12$, then  $s'\geq 12>\frac{4et}{t+3-4e}$ for $t\geq 84$, and hence, $\frac{4e(t+s')}{s'(t+3)}<1$. Also, $\psi(s'):=\frac{4e(t+s')}{s'(t+3)}$ is strictly decreasing in $s'$, since $\frac{\psi(s'+1)}{\psi(s')}=\frac{(t+s'+1)s'}{(t+s')(s'+1)}<1$. Hence, for $s'\geq 12$,
% \begin{equation}\label{eq:h(v,s',p)}
% h\Big(2t,s', \frac{2}{t+3}\Big)\leq  0.53+ \Big(\frac{4e(t+12)}{12(t+3)}\Big)^{12}<1.2.
% \end{equation}
% On the other hand, if $0\leq s'<12$, then, using that $h(2t,0,\frac{2}{t+3}) < 1.53$ we have 
where $\psi_a = 4^{s'}/s'!$ and $\psi_b = \frac{(t+s)}{t+3}\frac{t+s-1/2}{t+3} \dots \frac{t+ (s+1)/2}{t+3}$. 
Now $\psi_a = 64/6 < 10.7$ for $s' = 3,4$ and is otherwise less than $8.54$. 
On the other hand, $\psi_b$ is less than $1$ for $s' \leq 3$, and is decreasing in $t$ for $s' \geq 4$.
Thus for $s' \leq 3$, $\psi_a\psi_b < 10.7$. 
Using its value at $t = 100$ to bound $\psi_b$ for $s' = 4, \dots 25$ we get that $\psi_a\psi_b < 10.77$.
As $\psi_b < \genfrac(){}{}{t+s'}{t+3}^{s'} < e^{\frac{s'(s'-1)}{t+3}} < e^{s'}$, we have for $s' > 25$ that
$\psi_a\psi_b < (4e)^{s'}/s'! < 4e^{25}/25! < 6$. So for $s' \geq 0$, we get
\begin{equation}\label{eq:h(v,s',p)_2}
h\Big(2t,s', \frac{2}{t+3}\Big) < .53 + 10.77 = 11.3.   
\end{equation}

% \begin{equation}\label{eq:h(v,s',p)_2}
% h\Big(2t,s', \frac{2}{t+3}\Big) 
%  \leq  0.53+  \max_{0\leq s'\leq 11} \( \binom{2t+2s'}{s'}\Big(\frac{2}{t+3}\Big)^{s'} \)
% %  =    0.53 + \frac{4^{s'}}{s'!} \frac{(t+s')(t^{s'} + frac{3(s')^2 + s'}{4} t^{s' - 1} + o(t
% % \leq  0.53 + \max_{1\leq s'\leq 11} \( \frac{4e(t+s')}{s'(t+3)} \)^{s'}
% % \leq  0.53 + \max_{1\leq s'\leq 11} \( \frac{4(t+s')}{(t+3)}\)^{s'}e
% % \leq 11.3,
% \end{equation}
% where we get the last inequality by observing that $(t+s')/(t+3)$ is at most $1$ for $s' \leq 3$ and is 
% decreasing in $t$ for $s' \geq 4$. 

Therefore, we have that
\begin{equation*}
h(u,s,p)h(v,s',p)< h\Big(t,s,\frac{2}{t+3}\Big)h\Big(2t,s',\frac{2}{t+3}\Big)\overset{\eqref{eq:h(u,s,p)},\eqref{eq:h(v,s',p)_2}}{<} 0.08\cdot 11.3<0.99,
\end{equation*}
which yields~\eqref{eq:h_u,v}.
\end{proof}

\subsection{Intermediate values of $s$}\label{sub:medium_s} 

The remaining cases of $(s,s')$ are $0 \leq s' \leq  s \leq 9$.
In this section we deal with all but five of these.
We do this by showing the monotonicity of $g(s,s')$ on several ranges,
and then bounding $g(s,s')$ for four particular cases.  
% such as $g(s,s')<g(s-1,s'-1)$ or $g(s,s')<g(s-1,s')$. 
See Figure~\ref{fig:flows} for a schematic of the proof. 
In Claim \ref{clm:monotonicity} we show $g(s,s')<g(s-1,s'-1)$
for values of $(s,s')$ as indicated in the figure, (actually
for more values, but we only use those indicated in the figure). 
In Claims \ref{clm:g(s,1)}
and \ref{clm:g(s,0)} we show that $g(s,s') < g(s-1,s')$ for the values $s = 0$ and $1$
as indicated. 
In Claim \ref{clm:g(3,2)} we show that $g(s,s') < 0.99\(\mup(\cF^t_1)\)^2$
in the cases that $(s,s') =  (3,3),(3,2),(3,1)$ and $(2,0)$.

The final five values of $(s,s')$, the empty dots, 
 are dealt with in Section \ref{sec:non diagonal cases}.

%Next, we consider the cases $(s,s')$ with $2\leq s'\leq s\leq 9$. We will show monotonicity properties such as $g(s,s')<g(s-1,s'-1)$ or $g(s,s')<g(s-1,s')$. See Figure~\ref{fig:flows} for the details.

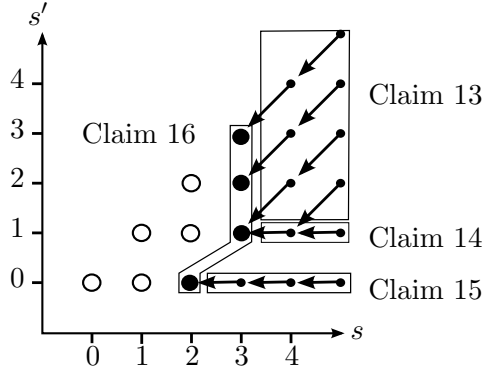
\begin{figure}\begin{center}
 \setlength{\unitlength}{300bp}%
 \begin{picture}(1,0.62499996)%
 \put(0,0){\includegraphics[width=\unitlength]{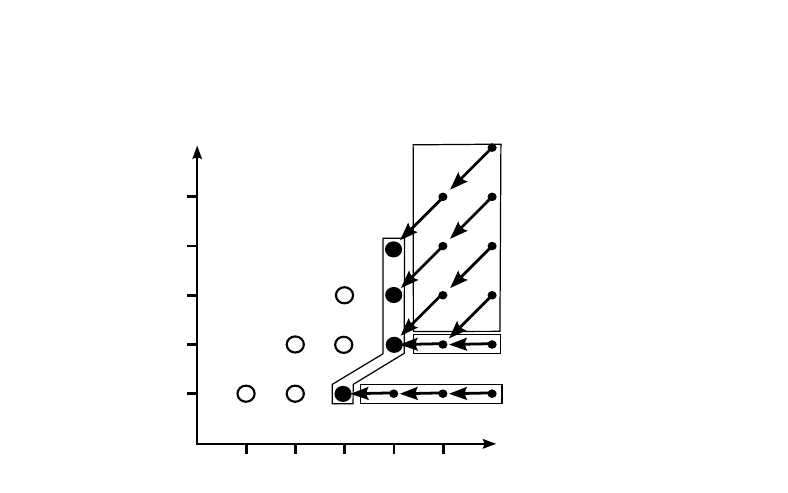}}%
    \put(0.23212717,0.449109){\makebox(0,0)[lb]{\smash{$s'$}}}%
    \put(0.6375274,0.05448193){\makebox(0,0)[lb]{\smash{$s$}}}%
    \put(0.305,0.021){\makebox(0,0)[lb]{\smash{$0$}}}%
    \put(0.367,0.021){\makebox(0,0)[lb]{\smash{$1$}}}%
    \put(0.430,0.021){\makebox(0,0)[lb]{\smash{$2$}}}%
    \put(0.492,0.021){\makebox(0,0)[lb]{\smash{$3$}}}%
    \put(0.555,0.021){\makebox(0,0)[lb]{\smash{$4$}}}%
    \put(0.21,0.11962739){\makebox(0,0)[lb]{\smash{$0$}}}%
    \put(0.21,0.18212724){\makebox(0,0)[lb]{\smash{$1$}}}%
    \put(0.21,0.2446273){\makebox(0,0)[lb]{\smash{$2$}}}%
    \put(0.21,0.30712736){\makebox(0,0)[lb]{\smash{$3$}}}%
    \put(0.21,0.36962742){\makebox(0,0)[lb]{\smash{$4$}}}%
    \put(0.66,0.35){\makebox(0,0)[lb]{Claim \ref{clm:monotonicity}}}%
    \put(0.66,0.17){\makebox(0,0)[lb]{Claim \ref{clm:g(s,1)}}}%
    \put(0.66,0.105){\makebox(0,0)[lb]{Claim \ref{clm:g(s,0)}}}%
    \put(0.30,0.30){\makebox(0,0)[lb]{Claim \ref{clm:g(3,2)}}}%
  \end{picture}%

 % \begin{picture}(200,250)(0,0)
 % \put(0,0){\includegraphics[scale=0.18]{flows3.pdf}}
 % \put(66,122){\fbox{ \parbox{3.7cm}{\vspace{3.5cm}\hspace{3.7cm} }}}
 % \put(185,120){Claim \ref{clm:monotonicity}}
 % \put(85,60){\fbox{ \phantom{ t t t t t t t t t t } } Claim \ref{clm:g(3,2)}}
 % \put(38,30){\fbox{ \phantom{ t t t t t t t t t t t t t t t t  } } Claim \ref{clm:g(s,1)}}
 % \put(38,2){\fbox{ \phantom{ t t t t t t t t t t t t t t t t } } Claim \ref{clm:g(s,0)}}
 % \end{picture}
 \end{center}
\caption{$(s_1,s_1')\rightarrow (s_2,s_2')$ means $g(s_1,s_1')<g(s_2,s_2')$.}
\label{fig:flows}
\end{figure}

\begin{claim}\label{clm:monotonicity} 
For $t \geq 10$ and $2\leq s' \leq s \leq 9$ with $(s,s') \neq (2,2)$ 
  we have $$g(s,s')<g(s-1,s'-1).$$
\end{claim}

\begin{proof} 

Recall that
$g(s,s')=f(u,s,p)f(v,s',p)$ and 
$g(s-1,s'-1)=f(u,s-1,p)f(v,s'-1,p).$
Hence, it suffices to show that
\begin{equation}\label{eq:f_u}
f(u,s,p)<f(u,s-1,p)
\end{equation}
and 
\begin{equation}\label{eq:f_v}
f(v,s',p)<f(v,s'-1,p).
\end{equation}

First, inequality~\eqref{eq:f_u} is equivalent to 
\begin{equation*}\label{eq:f_u_2}
\binom{u+2s}{s}\frac{pq}{u+s+1}< \binom{u+2s-2}{s-1}\frac1{u+s}.
\end{equation*}
We have that
\begin{align*}
\(\binom{u+2s}{s}\frac{pq}{u+s+1}\)\Big/\(\binom{u+2s-2}{s-1}\frac1{u+s}\)&\leq \frac{(u+2s)(u+2s-1)\cdot 2(t+1)}{s(u+s+1)\cdot (t+3)^2}\\
&=\frac{2(t+s+s')(t+s+s'-1)(t+1)}{s(t+s'+1)(t+3)^2} \\
& <\frac{2(t+s+s')(t+s+s'-1)}{s(t+s'+1)(t+3)} .
\end{align*}
This is decreasing in $t$ as $s + s' \geq 4$, so setting $t = 10$ and computing casewise, we get that it is less than 
 $.88$ for $2 \leq s' \leq s \leq 9$ and $(s',s) \neq (2,2)$.

Similarily, inequality~\eqref{eq:f_v} is equivalent to 
\[
\binom{v+2s'}{s'}\frac{pq}{v+s'+1}< \binom{v+2s'-2}{s'-1}\frac1{v+s'}.
\]
We have that
\begin{align*}
\(\binom{v+2s'}{s'}\frac{pq}{v+s'+1}\)\Big/ \(\binom{v+2s'-2}{s'-1}\frac1{v+s'}\)&\leq\frac{2(v+2s')(v+2s'-1)(t+1)}{s'(v+s'+1)(t+3)^2}\\
&<\frac{2(t+s+s')(t+s+s'-1)}{s'(t+s+1)(t+3)}.
\end{align*}
This is again decreasing in $t$, and with $t = 10$ we compute that it is less than $.88$ 
for $2 \leq s' \leq s \leq 9$ and $(s',s) \neq (2,2)$.  (The maximum value is at $(s,s') = (3,3)$, which
 is why it is the same value as above.)

\end{proof}

\begin{claim}\label{clm:g(s,1)} For $s\geq 2$ and $s'=1$, we have 
$g(s,1)< g(1,1)$.
\end{claim}
\begin{proof}
 Note that $u=t-s+1$ and $v=t+s-1$.
 Recalling~\eqref{eq:f(ell,i,p)}, we can write 
\begin{align*}
 p^s f(u,s,p)&=C_1q^s+C_2h(s)q^s, \hskip 1em \mbox{and}\\
 p^{-s}f(v,1,p)&=C_3q^{-s}+C_4(t+s),
\end{align*}
where $h(s):=\binom{t+s+1}s(t-s+2)p^s$ and
 $$C_1=\alpha^{t}, \hskip 1em C_2=\frac{p^{t-1}(1-\alpha)}{t+3}, \hskip 1em C_3=\alpha^{t}, \hskip 1em\mbox{and}\hskip 1em C_4=p^tq(1-\alpha).$$
 Note that 
  $C_1, C_2, C_3, C_4>0$ depend only on $t$ and $p$ (and do not depend on $s$).

Multiplying $p^s f(u,s,p)$ and $p^{-s}f(v,1,p)$, we have that
\begin{equation}\label{g(s,1)}
g(s,1)=D_1+D_2q^s(t+s)+D_3h(s)+D_4h(s)q^s(t+s), 
\end{equation}
where $D_1, D_2, D_3, D_4>0$ depend only on $t$ and $p$. % (and do not depend on $s$).

We claim that $g(s,1)$ is strictly decreasing in $s$ for $s\geq 1$. By~\eqref{g(s,1)}, it suffices to show that $q^s(t+s)$ and $h(s)$ are strictly decreasing in $s$. First, $q^s(t+s)$ is strictly decreasing in $s$ since
$$\frac{q^{s+1}(t+s+1)}{q^s(t+s)}=\frac{q(t+s+1)}{(t+s)}\leq \frac{t(t+s+1)}{(t+1)(t+s)}<1,
$$ where the first inequality follows from $q\leq t/(t+1)$ and the last inequality holds for $s\geq 1$.
Next, $h(s)$ is strictly decreasing in $s$ since
\begin{equation*}
\frac{h(s+1)}{h(s)}\leq  \frac{2(t+s+2)(t-s+1)}{(s+1)(t-s+2)(t+3)}< \frac{2(t+s+2)}{(s+1)(t+3)}\leq 1,
\end{equation*}
where the first inequality follows from $p\leq 2/(t+3)$ and the last inequality follows from $s\geq 1$.
\end{proof}

\begin{claim}\label{clm:g(s,0)} For $s\geq 2$ and $s'=0$, we have 
$g(s,0)< g(1,0)$.
\end{claim}
\begin{proof}
Again, noting this time that $u=t-s$ and $v=t+s$, we write %Note that $u=t-s$ and $v=t+s$. Recalling~\eqref{eq:f(ell,i,p)}, we can write
\begin{align*}
p^s f(u,s,p)&=C_1q^s+C_2h(s)q^s, \hskip 1em \mbox{and}\\
p^{-s}f(v,0,p)&=C_3q^{-s}+C_4,
\end{align*}
where  $h(s):=\binom{t+s}s(t-s+1)p^s$, and $C_1, C_2, C_3, C_4>0$ (different from above)
depend only on $t$ and $p$. % (and do not depend on $s$). 
Multiplying $p^s f(u,s,p)$ and $p^{-s} f(v,0,p)$, we have that
\begin{equation}\label{g(s,0)}
g(s,0)=D_1+D_2q^s+D_3h(s)+D_4h(s)q^s, 
\end{equation}
where $D_1,D_2,D_3,D_4>0$ depend only on $t$ and $p$. % (and do not depend on $s$). 

We claim that $g(s,0)$ is strictly decreasing in $s$ for $s\geq 1$. By~\eqref{g(s,0)}, it suffices to show that $h(s)$ is strictly decreasing in $s$. Indeed,
\begin{equation*}
\frac{h(s+1)}{h(s)}\leq \frac{2(t+s+1)(t-s)}{(s+1)(t-s+1)(t+3)}<\frac{2(t+s+1)}{(s+1)(t+3)}\leq 1,
\end{equation*}
where the first inequality follows from $p\leq 2/(t+3)$ and the last inequality holds for $s\geq 1$.
\end{proof}

\begin{claim}\label{clm:g(3,2)}  For $t \geq 52$ and  $(s,s') = (3,3),(3,2),(3,1)$ and $(2,0)$ we have $g(s,s')<0.99\(\mup(\cF^t_1)\)^2.$
\end{claim}
\begin{proof}
 We give the calculations for the case $(s,s') = (3,1)$.  The calculations for
 the other cases are  very similar, and given in the appendix. 
For the estimation in all cases we use $e^{-2} \leq q^t$, and 
$q^{-i}=(\frac{t+3}{t+1})^i<2$ for $1\leq i\leq 6$ and $t\geq 16$.
 Noting that $u = t-2$ and $v = t+2$ we get that 
 \begin{align*}
  f(u,3,p) = \alpha^{t-1} + \binom{t+4}{3}\frac{t-1}{t+2}p^{t+1}q^3(1 - \alpha) < e^2\frac{p^{t-1}}q + \frac{(t+4)(t+3)(t-1)}{6}p^{t+1}q^3,
\end{align*}  
% where we have used that $e^{-2} \leq q^t \leq 1/2$, 
and we get 
 \begin{align*}
  f(v,1,p) = \alpha^{t+3} + (t+4)\frac{t+3}{t+4}p^{t+3}q(1 - \alpha) < \frac{e^2p^{t+3}}{q^3} + (t+3)p^{t+3}q.
\end{align*}  
 Thus as $(\mup(\cF^t_1))^2 > (t+2)^2p^{2t+2}q^2$ we get that
 \begin{align*}
    \frac{g(3,1)}{(\mup(\cF^t_1))^2} & < \frac{e^4}{q^6(t+2)^2} + \frac{e^2(t+3)}{(t+2)^2q^2} +
    \frac{e^2p^2(t+4)(t+3)(t-1)}{6q^2(t+2)^2} + \frac{p^2q^4(t+4)(t+3)^2}{6(t+2)^2}\\
                         & < \frac{2e^4}{(t+2)^2} + \frac{2e^2(t+3)}{(t+2)^2} +
    \frac{8e^2(t+4)(t+3)(t-1)}{6(t+2)^2(t+3)^2} + \frac{4(t+4)(t+3)^2}{6(t+2)^2(t+3)^2}.
 \end{align*} 
 This is less than $.99$ for $t \geq 28$. 

\end{proof}

Referring to Figure \ref{fig:flows}, or our outline of the proof preceding Claim \ref{clm:monotonicity}, Claims~\ref{clm:s_large}--\ref{clm:g(s,0)} imply the following corollary.

% Claims~\ref{clm:s_large} and~\ref{clm:s geq s' geq 3}--\ref{clm:g(s,0)} imply the following corollary.
\begin{corollary}\label{coro:easy_cases} If $\mup(\cA)\mup(\cB)\leq  \(\mup(\cF_1^t)\)^2$ holds for
\[
 (s,s')=(0,0),(1,0),(1,1),(2,2),
\] 
then, for all $(s,s')$ other than  $(0,0),(1,0),(1,1),(2,2)$,
$$\mup(\cA)\mup(\cB)<  \(\mup(\cF_1^t)\)^2.$$ \end{corollary}

 %%%%%%%%%%%%%%%%

\section{Remaining cases}\label{sec:non diagonal cases}

\subsection{Definitions} We introduce several definitions and notation. For $A\subset [n]$, let $(A)_i$ be the $i$-th smallest element of $A$. For $A,B\subset [n]$, we say that \emph{$A$ shifts to $B$}, denoted by 
 \begin{equation*}%\label{eq:arrow}
     A\shiftsto B,
 \end{equation*}
 if $|A|\leq |B|$ and $(A)_i\geq (B)_i$ for each $i\leq |A|$. In other words, as walks on a two-dimensional grid, each edge of the walk $B$ is not contained in the area to the right of the walk $A$. For example, $\{2,4,6\}\shiftsto \{1,4,5,7\}.$
\begin{fact}[\cite{F}, Fact 2.8]\label{fact:shift_1}
Let $\cF$ be a shifted, inclusion maximal family in $2^{[n]}$. If $F\in \cF$ and $F\shiftsto F'$, then $F'\in \cF$.
\end{fact}
This immediately implies the following. 
\begin{fact}\label{fact:shift_2}
Let $\cF$ be a shifted, inclusion maximal family in $2^{[n]}$. If $F'\not\in\cF$, then every $F\in \cF$ satisfies $F\not\shiftsto F'$.
\end{fact}

For $t\in [n]$ and $F\subset [n]$, the \emph{dual} of $F$ with respect to $t$ is defined by
\begin{equation*}%\label{eq:dual}
   \dual_t(F):=[(F)_t-1]\cup ([n]\setminus F).
\end{equation*}
Viewed as walks on a two-dimensional grid, the walk $\dual_t(F)$ is obtained by reflecting $F$ across the line $y=x+(t-1)$ and ignoring the part $x<0$. (See Figure~\ref{fig:dual}.)
\begin{figure}\begin{center}
\includegraphics[scale=0.18]{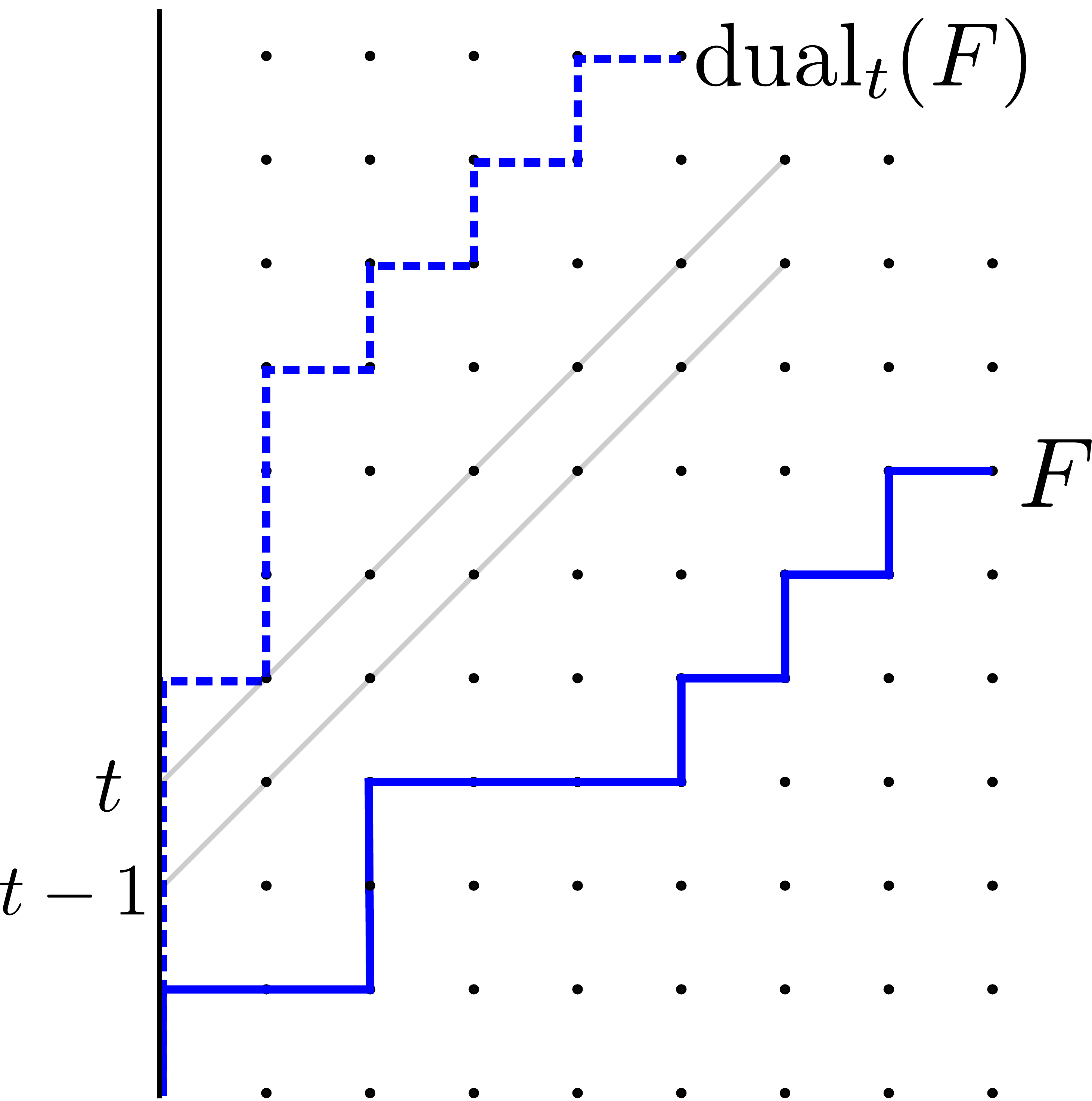}
\end{center}
\caption{The walk $\dual_t(F)$ of $F$ with respect to $t$}
\label{fig:dual}
\end{figure}

The dual $\dual_t(F)$  of a set $F$ is defined so that its intersection with $[t,n]$
is the complement of that of $F$, so $|F \cap \dual_t(F)| < t$.  
This gives the following.  

\begin{fact}[\cite{F}, Fact 2.9]\label{fact:dual}
Let $\cA$ and $\cB$ be cross $t$-intersecting families. If $A\in \cA$, then
  $\dual_t(A)\not\in \cB$.
\end{fact}

 For integers $\ell, i\geq 1$ and $s\geq0$, let
\begin{align}\label{eq:D}
D^\ell_s(i):=&[\ell-1]\cup \{\ell-1+2, \ell-1+4, \dots, \ell-1+2s\}  \nonumber \\ &
 \cup \{\ell + 2s\} \cup \{\ell+2s+i+2k\in [n] : k=1,2,\dots\}.
\end{align}
This walk is the maximally shifted walk in $\hF^{\ell}\cap\cF_s^{\ell}$ with the
property that it goes left for $i+1$ steps after hitting the line 
$y = x + \ell$ at $(s,s+l)$,
and then after that does not go above the line $y = x + \ell - i$. 
Note that $D^\ell_s(i)=D^\ell_s(n-\ell-2s-1)$ for $i\geq n-\ell-2s-1$, and hence, we assume that 
\begin{equation}\label{eq:range_i}1\leq i\leq n-\ell-2s-1.\end{equation}
The walks $D^{t-1}_1(i)$ and $D^{t+1}_0(j)$ are denoted by $D^{\cA}_i$ and $D^{\cB}_j$
 in~\cite{F}. (They are depicted in Figure 1 of~\cite{F}.)

\subsection{The cases  $(s,s')=(2,1)$ and $(1,0)$.}

Note that in the cases $(s,s')=(2,1)$ and $(1,0)$ we have that $u=t-1$ and $v=t+1$. 
%In our proof of Lemma~\ref{lem:(2,1)} and even more so in our proof of 
%Lemma~\ref{lem:(1,0)},
%we avoid tedious calculations by appealing to estimations done in the
% proof of  Lemma 3.5 in~\cite{F}.
%Many of these estimations, like Lemma \ref{lem:epsilon} contain a value 
%$\epsilon > 0$ that can be taken 
%arbitrarily small, for fixed $t$, by assuming that $n$ is sufficiently large,
%(where how large `sufficiently' is usually depends on $t$). 
%We will frequently do this; recall from our setup that this is permitted. 

\begin{lemma}\label{lem:(2,1)}
Let $t\geq 42$.
For $(s,s')=(2,1)$, we have $\mup(\cA)\mup(\cB)< \(\mup(\cF^t_1)\)^2$. 
\end{lemma}

\begin{proof}

Consider the following  cases of the walks defined in \eqref{eq:D}.
%For  $1\leq i\leq n-t-4=:i_{\max}$ let  
For  $1\leq i\leq n-t-4$, let
   $$D^{t-1}_2(i)=[t-2]\cup \{t,t+2\} \cup \{t+3\} \cup \{t+5+i, t+7 +i, t+9+i, \dots\}\in 
\hF^{t-1}\cap\cF_2^{t-1},$$
% and for $1\leq j\leq n-t-4=: j_{\max}$, let
and for $1\leq j\leq n-t-4$, let
  $$D^{t+1}_1(j)= [t]\cup \{t+2\}\cup\{t+3\} %\{t+2,t+3\} 
\cup \{t+5+j, t+7+j, t+9+j\dots\} \in 
\hF^{t+1}\cap\cF_1^{t+1}.$$

By Fact~\ref{fact:shift_1} and the fact that
  $\hA\neq\emptyset$ and $\hB\neq\emptyset$, we have that $D^{t-1}_2(1)\in \cA$
 and $D^{t+1}_1(1)\in\cB$. So 
%$\{ i : D^{t-1}_2(i)\in \cA\}$ and $\{j:D^{t+1}_1(j)\in\cB\}$ are not empty, and
 the following positive integer values are well defined:
 $$I:=\max\{ i : D^{t-1}_2(i)\in \cA\} \hskip 1em \mbox{and} \hskip 1em J:=\max\{j:D^{t+1}_1(j)\in\cB\}.$$

We start with the following general bounds on $\mu_p(\cA)$ and $\mu_p(\cB)$,
we then show, with casework depending on $I$ and $J$, that they are sufficient.

\begin{claim}\label{cl:22A}
Let $t\geq 20$.
For every $\epsilon>0$ the following
holds for sufficienlty large $n$:
\[ \mu_p(\cA)/p^t  < a_1(p,t)+a_2(p,t)\alpha^{J-1} - a_3(p,t)q^{I-2} + \epsilon,\]
  where 
\begin{align*}
 a_1(p,t)&:=1 + tpq + \frac{t(t+3)}{2}pq^2,\\
 a_2(p,t)&:=q^{-t} - 1 - tpq - \frac{t(t+3)}{2}p^2q < 5, \\ 
 a_3(p,t)&:=\frac{(t+2)(t-1)}2pq^5(1-\alpha) >
           \frac{pq^2}2(t^2-7t).
\end{align*}
\end{claim}
\begin{proof}\claimproof
Let $\epsilon>0$ be given and let $\delta=\epsilon/a_1(p,t)$.
As $s = 2$ we have that $\mu_p(\cA)=\mu_p(\tA)+\mu_p(\cA_2)$. 
To bound $\mu_p(\tA)$ observe that since $D^{t+1}_1(J)\in\cB$, % where
%\[
%D^{t+1}_1(J)=[t]\cup\{t+2,t+3\}\cup\{t+J+5,t+J+7,\ldots\}\in\cB, 
%\]
its dual walk
\[
\dual_t(D^{t+1}_1(J))=[t-1]\cup\{t+1\}\cup[t+4,t+J+4]\cup\{t+J+6,t+J+8,\ldots\}
\]
is not in $\cA$, and all walks in $\cA$ must cross it, 
which is equivalent to hitting $Q_0: = (0,t)$, $Q_1:=(1,t+1)$
or $Q_2:= (2,t+1)$, or hitting the line $L: y=x+(t+J-1)$. Further, walks in 
$\tA\subset\cF^{u}=\cF^{t-1}$ all hit the line $L': y=x+t$.
So we have
\begin{align*} 
\mu_p(\tA) &\leq\mu_p(\mbox{walks in $\tA$ hitting $L$})\\
&\quad+\mu_p(\mbox{walks in $\tA$ not hitting $L$ but hitting $Q_0$}) \\
&\quad+\mu_p(\mbox{walks in $\tA$ not hitting $L$ or $Q_0$ but hitting 
$Q_1$})\\
&\quad+\mu_p(\mbox{walks in $\tA$ not hitting $L$, $Q_0$ or
$Q_1$ but hitting $Q_2$ and $L'$}).
\end{align*}
Using Lemma \ref{lem:mu_F} for the first and last line, and Lemma \ref{lem:epsilon}
for the last three, this gives the following, 
\begin{align} 
\mu_p(\tA) &\leq  \alpha^{t+J-1}+p^t(1-\alpha^{J-1}+\delta) 
+tp^{t+1}q(1-\alpha^{J-1}+\delta) \nonumber\\
&\quad+\frac{t(t+3)}{2}p^{t+1}q^2(\alpha-\alpha^J
+\delta/2).\label{eq:tA}
\end{align}
For the last line we also used that there are 
$\binom{t+3}2-\binom{t+3}1=\frac{t(t+3)}2$ ways of walks from 
$(0,0)$ to $Q_2$ that do not touch the line $L'$. In fact, 
Lemma 2.13(ii) of \cite{F} tells us that the 
number of walks from $(0,0)$ to $(x_0,y_0)$ not hitting the line $y = x + c$ is
\begin{equation}\label{eq:la2.13ii}
\binom{x_0+y_0}{x_0} - \binom{x_0+y_0}{y_0 - c}  
\end{equation}
for $0 < c < y_0 < x_0+c$.

Now we bound $\mup(\cA_2)$. Recall from Lemma~\ref{lem:structure} and 
\eqref{def:bF} that $\cA_2:=\hA\sqcup\dhA\subset \bF^{t-1}_2$.
Any walk in 
$\bF_2^{t-1}=\cF^{t-1}_2\setminus\tF^{t-1}$ hits $Q_2$ without 
hitting the line $L'$,  so without hitting $Q_0$ or $Q_1$, and then 
continues on without hitting $L'$. 
So we have 
\begin{equation}\label{eq:hA}
\mu_p(\bF_2^{t-1})\leq \frac{t(t+3)}{2}p^{t+1}q^2(1-\alpha+\delta/2).
\end{equation}

On the other hand, as $D_2^{t-1}(I+1)\not\in\cA_2$, $\cA_2$ contains no 
walks in
 \[ \cW = \{W\in\hF^{t-1}\cap\cF_2^{t-1}:W\to D^{t-1}_2(I+1)\}. \]
Such walks hit $(2,t)$ without hitting the line $y = x + (t-1)$, 
then hit $(2,t+1)$ and then $(I+3,t+1)$ on the line
$y = x + (t - I - 2)$. After that, they never hit the 
line $y = x + (t - I - 1)$. 
Using \eqref{eq:la2.13ii} for $x_0=2,y_0=t,c=t-1$ we have  
\begin{equation}\label{eq:eA}
  \mu_p(\cW)  \geq\(\binom{t+2}2-\binom{t+2}1\)p^{t+1}q^{I+3}(1-\alpha) = 
   a_3(p,t)p^tq^{I-2}.
\end{equation}

We now combine \eqref{eq:tA}, \eqref{eq:hA} and \eqref{eq:eA} using the fact that
$\cA = \tA \cup \cA_2 \subset \tA \cup (\bF_2^{t-1} \setminus \cW)$. Observing how nicely
\eqref{eq:hA} combines with the last term in~\eqref{eq:tA}, we get
\[
\mu_p(\cA) / p^t 
\leq \frac{\alpha^{J-1}}{q^t} + (1 - \alpha^{J-1} + \delta)
(1 + tpq ) +\frac{t(t+3)}{2}pq^2(1-\alpha^J+\delta)
-  a_3(p,t)q^{I-2}.
\]
Rearranging this we get 
\begin{equation*}%\label{eq:cl22A}
\mup(\cA) / p^t \leq (1+\delta)a_1(p,t)+a_2(p,t)\alpha^{J-1} - a_3(p,t)q^{I-1},
\end{equation*}
which is equivalent to the statement of the claim.

To get the bound on $a_2(p,t)$,  %observe that $q^{-t} \leq ( (t+3/(t+1) )^t < e^2$
recall from \eqref{eq:q^t} that $q^{-t} < e^2$
 and observe that the other terms in $a_2(p,t)$ are decreasing in $p$, so
 for $t \geq 20$ we have
\[  a_2(p,t) \leq a_2\(\frac1{t+1},t\) < e^2 - 1 - 
 \frac{t^2(3t+5)}{2(t+1)^3} 
  < e^2 - 1- 1.4 < 5. \] 

To get the bound on $a_3(p,t)$, observe that $q^3(1 - \alpha)$ is decreasing in 
$p$, so  letting $p=\frac2{t+3}$ it follows that 
\[
 q^3(1-\alpha)\geq\(\frac{t+1}{t+3}\)^3\(\frac{t-1}{t+1}\)
=\frac{(t-1)(t+1)^2}{(t+3)^3}
\]
and
\begin{align*}
  a_3(p,t) &= \frac{pq^2}2 \( (t+2)(t-1)q^3(1-\alpha) \)
\geq \frac{pq^2}2 (t+2)(t-1) \frac{(t-1)(t+1)^2}{(t+3)^3}\\
     & = \frac{pq^2}2 \( \frac{t^5 + 2t^4 - 2t^3 - 4t^2 + t + 2}{t^3 + 9t^2 + 27t + 27} \) 
%     &  = \frac{pq^2}2 \( t(t+3) - 6t\frac{t^3 + \frac52t^2 + \frac43t + \frac12}{t^3 + 3t^2 + 3t + 1} \)\\
      > \frac{pq^2}2(t^2-7t),
 \end{align*}
 which gives the bound. 
\end{proof} % proof of claim

Similarily, we get the following. The proof is in the  appendix. 

\begin{claim}\label{cl:23B}
 For every $\epsilon>0$ the following holds for $t \geq 18$:
  \[ \mup(\cB) / p^{t+2} \leq  b_1(p,t)+b_2(p,t)\alpha^{I-1}- b_3(p,t)q^{J-2} + \epsilon, \] 
where %$b_1(p,t):=1 + (t+2)q$, $ b_2(p,t):=q^{-(t+2)} -1 - (t+2)p < 4.5$, 
%and $b_3(p,t):=(t+1)q^3(1-\alpha)$.
\begin{align*}
 b_1(p,t)&:=1 + (t+2)q, \\
 b_2(p,t)&:=q^{-(t+2)} -1 - (t+2)p < 4.5, \\
 b_3(p,t)&:=(t+1)q^4(1-\alpha) > (t-7)q. 
\end{align*}
\claimproof \qed
\end{claim}

To prove the lemma it is now  enough to show that
     \[ (\mup(\cA) / p^t)(\mup(\cB) / p^{t+2}) < z^2, \]
where
\begin{equation}\label{eq:def of z}
z = \mup(\cF^t_1)/p^{t+1}= t+2 - (t+1)p.   
\end{equation}
 We have cases depending on $I$ and $J$.
% The strategy is roughly as follows.  
% The positive terms $a_1$ and $b_1$
% in the claims increase in $p$, as does the $\alpha$ multiplying $a_2 < 2.4$ and 
% $b_2 < 4.5$. On the other hand, $z$ decreases in $p$, so ignoring the $a_3$
% and $b_3$ terms, the inequality is tightest at $p = 2/(t+3)$.
% For this $p = 2/(t+3) $ we have approximately $a_1(p,t)b_1(p,t) = z^2 -1$.
% When $I,J \geq 3$, $\alpha^{I-1}$ and $\alpha^{J-1}$ are small enough 
% that we do not need the terms $a_3$ and $b_3$. 
% When either of $I$, or $J$, is at most $2$, these bounds are not enough, 
% but then the $a_3$ or $b_3$ terms become useful.

\smallskip\noindent{$\bullet$ \bf Case 1.} \textit{Suppose $I\geq 3$ and $J\geq 3$.}

  First observe that for $J \geq 3$ we have 
  \begin{equation}\label{eq:A}
   \mu_p(\cA)/p^t < a_1(p,t)+5\alpha^{2}.   
  \end{equation}
  Indeed if $J = 3$ this is immediate from  Claim \ref{cl:22A}
  by taking $\epsilon < (5 - a_2(p,t))\alpha^2$.  
  For $J \geq 4$, Claim \ref{cl:22A} gives that 
  $\mu_p(\cA)/p^t<a_1(p,t)+5\alpha^{3}+ \epsilon$. Because 
  $\alpha^3 < \alpha^2$, the claim follows by taking 
  $\epsilon < 5(\alpha^{2} - \alpha^{3})$. 
  Similarily, it follows from Claim \ref{cl:23B} that for $I \geq 3$ and
  $t \geq 18$ 
  we have
  \begin{equation*}%\label{eq:B}
    \mu_p(\cB)/p^{t+2} < b_1(p,t)+4.5\alpha^{2}. 
  \end{equation*}
% \begin{align*}
% \mu_p(\cB) &= \mu_p(\tB)+\mu_p(\cB_1) 
% \leq  p^{t+2}(1+4.5\alpha^{I-1})+(t+2)p^{t+2}q-\delta', 
% \end{align*}
% where
% \[
% \delta'=p^{t+2}
% \left(
% \alpha^{I-1}\left(
% (t+2)p-q^{-t-2}+5.5\right)
% -\epsilon(2(t+2)q+1)
% \right).
% \]
% Notice that $b(p):=(t+2)p-q^{t+2}$ is decreasing in $p$, and
% $b(p)\geq b(\frac2{t+3})$, which is increasing in $t$, so 
% for $t\geq 20$ it follows that $b(\frac2{t+3})\geq b(\frac2{23})>-5.5$. 
% By taking $n$ sufficiently large and $\epsilon>0$ sufficiently small
% we may assume that $\delta'>0$. Thus we get
So it suffices to show $xy<z^2$ where
$x:=a_1(p,t)+5\alpha^2$,
$y:=b_1(p,t)+4.5\alpha^2$.
One can show that $y/z$ ins increasing in $p$ Clearly $x$ is increasing and 
$z$ is decreasing. One can also show that $y/z$ is increasing (see~\ref{app:othercomp}), so it is enough to check the inequality $xy-z^2<0$ at $p=\frac2{t+3}$.
By direct computation we see that this is true if $t\geq 42$.

%allow us to improve \eqref{eq:cl22A} to
%\begin{align*}
%\mu_p(\cA)/p^t &
% < (1 + \epsilon)a_1(p,t)+a_2(p,t)\alpha^{J-1} - \frac{(t+2)(t-1)}2pq^{I+3}(1-\alpha). \\
% & < a_1(p,t)+2.4\alpha - \frac{(t+2)(t-1)}2pq^{I+3}(1-\alpha) 
%\end{align*}

\smallskip\noindent{$\bullet$ \bf Case 2.} \textit{Suppose that $I=1$ or $2$.}

By Claim~\ref{cl:22A} we get
\begin{align*}
   \mup(\cA)/p^t & < a_1(p,t) + a_2(p,t)\alpha^{J-1} - a_3(p,t) + \epsilon \\
      & < a_1(p,t) + 5 - a_3(p,t) \\
      & < 1 + tpq + 5tpq^2 + 5 < 18.
\end{align*} 
 The last inequality uses that $pq$ and $pq^2$ are increasing in $p$, so $p$ can
 be taken as $2/(t+3)$. 
 By Claim~\ref{cl:23B} we have that 
 $\mup(\cB)/p^{t+2} < b_1(p,t) + 4.5 = z + q + 4.5 < z + 5$, and so 
 \[(\mup(\cA)/p^t)(\mup(\cB)/p^{t+2}) < 18(z + 5).  \] 
Since $18(z+5) < z^2$ if $18 \leq z-5$ we see that $z\geq 23$ suffices.
Since $z$ is minimized when $p=\frac2{t+3}$ and $z\geq t+\frac4{t+3}$,
it follows that $z\geq 23$ if $t\geq 23$.

 \smallskip\noindent{$\bullet$ \bf Case 3.} \textit{Suppose that $J=1$ or $2$.} 
 By Claim~\ref{cl:22A} we get that  
 \begin{align*}
 \mup(\cA)/p^t &< a_1(p,t) + a_2(p,t) < 1 + tpq + \frac{t(t+3)}2pq^2+5\\
                 &< 1 + 2 + t + 5 = t+8.% \red{3.9+t+\frac6{t+3} < t+5}.
 \end{align*}
  The third inequality uses that $pq$ and $pq^2$ are increasing in $p$ so 
   $p =2/(t+3)$ can be assumed. 

% Having finished the case that $I \leq 2$ we can assume that $I \geq 3$, from 
 From Claim~\ref{cl:23B} we get that 
  \begin{align*}
     \mup(\cB)/p^{t+2} & < b_1(p,t)+ b_2(p,t) - b_3(p,t) + \epsilon \\ 
         & < b_1(p,t) + 4.5 - b_3(p,t) < 14.5. 
      %   & < 1 + 9q + 4.5 < 14.5.
   \end{align*}
   So $(\mup(\cA)/p^t)(\mup(\cB)/p^{t+2}) < 14.5(t+8)$ which is less than 
   $z^2$ for $t \geq 23$. 
   
  This completes the proof for Case 3, and so for the lemma.
\end{proof}

 \begin{lemma}\label{lem:(1,0)}
Let $t\geq 26$.
For $(s,s')=(1,0)$, we have $\mup(\cA)\mup(\cB)< \(\mup(\cF^t_1)\)^2$. 
\end{lemma}

\begin{proof}

Again, consider the following particular cases of walks defined in 
\eqref{eq:D}. For $1\leq i\leq n-t-2$, let
$$D^{t-1}_1(i)=[t-2]\cup \{t\}\cup\{t+1\}\cup 
\{t+3+i,t+5+i,t+7+i,\ldots\}\in 
\hF^{t-1}\cap\cF_1^{t-1}.$$
For $1\leq j\leq n-t-2$, let
$$D^{t+1}_0(j)= [t+1]\cup 
\{t+3+j,t+5+j,t+7+j,\dots\} \in 
\hF^{t+1}\cap\cF_0^{t+1}.$$
%\blue{One can easily see that any walk in $\hB$ shifts to $D^{t+1}_0(1)$, and as
%  $s =1$ any walk in $\hA$  shifts to $D^{t-1}_1(1)$.} 
 Again the following values are well defined: $$I:=\max\{ i : D^{t-1}_1(i)\in \cA\} \hskip 1em \mbox{and} \hskip 1em J:=\max\{j:D^{t+1}_0(j)\in\cB\}.$$

 Analagous to Claims \ref{cl:22A} and \ref{cl:23B} we get the following two claims,
 which are proved in the appendix.
  
 \begin{claim}\label{cl:10A}
For every $\epsilon>0$ the following holds:
    \[ \mu_p(\cA)/p^t  < a_1(p,t)+a_2(p,t)\alpha^{J-1} - a_3(p,t)q^{I-1} + \epsilon, \]
  where 
\begin{align*}
 a_1(p,t):=1 + tq, \qquad  a_2(p,t):=q^{-t} < 7.4, \qquad
   a_3(p,t):=(t-1)q^3(1-\alpha) > (t-7)q.  
% a_3(p,t)&:=\frac{(t+2)(t-1)}2pq^4(1-\alpha).  
\end{align*}
 \end{claim}

 \begin{claim}\label{cl:10B}  
For every $\epsilon>0$ the following holds for $t \geq 20$:
  \[ \mup(\cB) / p^{t+2} \leq  b_1(p,t)+b_2(p,t)\alpha^{I-1}- b_3(p,t)q^{J-1} + \epsilon, \] 
where %$b_1(p,t):=1 + (t+2)q$, $ b_2(p,t):=q^{-(t+2)} -1 - (t+2)p < 4.5$, 
%and $b_3(p,t):=(t+1)q^3(1-\alpha)$.
\begin{align*}
 b_1(p,t):=1/p, \qquad b_2(p,t):=q^{-(t+2)} < 7.4, \qquad 
 b_3(p,t):=(q^2/p)(1-\alpha) > .75/p.    
\end{align*}
%\claimproof \qed
\end{claim}

Using these claims, we finish the lemma by considering three cases.

\smallskip\noindent{$\bullet$ \bf Case 1:} \textit{Suppose that $I\geq 2$ and $J\geq 2$.}
 As \eqref{eq:A} followed from Claim \ref{cl:22A} for $I,J \geq 3$ we have that 
 for $I,J \geq 2$ and $t \geq 20$, the following inequalities follow from Claims
  \ref{cl:10A} and  \ref{cl:10B}.
 \begin{gather*}
 \mup(\cA)/p^t < 1 + tq + 7.4\alpha =: a,\\
 \mup(\cB)/p^{t+2} < 1/p + 7.4\alpha =: b.
 \end{gather*}

We need to show that $ab/z^2 < 1$, where $z$ is defined in 
\eqref{eq:def of z}.
As $z>(t+2)q$ we show that $ab<((t+2)q)^2$, and
it is enough to show that $a<(t+2-0.335)q$ and $b<(t+2+0.335)q$.
The former is equivalent to $1+7.4\alpha<1.665q$, 
which is true for $p<0.069$. So it holds for $t \geq 26$. 
The latter is equivalent to $(1/p+7.4\alpha)/q<t+2.335$,
the left side of which is decreasing in $p$ for $p < .1$.
Evaluating it at $p=1/(t+1)$ we see that it too holds for $t\geq 26$.

%\begin{align*}
%  \frac{ab}{z^2} < \frac{(1 + tq + 7.4p/q)(1 + 7.4p^2/q)}{pq^2(t+2)^2} 
%  \leq  \frac{\(1 + \frac{t^2}{t+1} + 7.4\frac{2}{t+1}\)\(1 + 7.4\frac{4}{(t+1)(t+3)}\)}{t^2(t+2)^2/(t+1)^3} < 1. 
%\end{align*}
% The second inequality holds since $p, 1/q,$ and $pq^2$ are increasing
%  in $p$ for $p\leq 0.2$, and the last inequality holds for $t\geq 45$.

\smallskip\noindent{$\bullet$ \bf Case 2:} \textit{Suppose that $I=1$.}
From Claims \ref{cl:10A} and  \ref{cl:10B} we get that
\begin{align*}
 \mup(\cA)/p^t &< (1+tq) + a_2(p,t)\alpha^{J-1} - (t-7)q +\epsilon\\
&<1 + 7q + 7.4 =: a,\\
 \mup(\cB)/p^{t+2} &< 1/p + 7.4 =: b.
\end{align*}

Again we need to show that $ab/z^2 < 1$.
It is enough to show that $c:=ab/((t+2)q)^2 < 1$; and
indeed, $c$ is decreasing in $p$ so evaluating $c$ at $p=1/(t+1)$
we see that $c$ is at most $\frac{7 (t+1) (5 t+42) (11t+6)}{25 t^2 (t+2)^2}<1$
for $t\geq 20$.

% and $b_1/b = \frac{ (t+3)/2 + 7.4}{(t+3)/2} < \frac{24 + 7.4}{24} < 4/3$,
% we get that $a_1b_1 < (4/9)ab < z^2$
% as needed. }  
 % Indeed, for $t \geq 45$ we 
 % have that $(1/p + 7.4) < (4/3)(1/p + 7.4\alpha)$ and that 
 % $(1 + 7q + 7.4) <  (3/4)(1 + tq + 7.4\alpha)$, so 
 % \begin{align*}
 %  a_1b_1 = (1 + 7q + 7.4)(1/p + 7.4) < (1 + 7q + 7.4)(1/p + 7.4\alpha)(4/3) <
 %   (1 + tq + 7.4\alpha)(1/p + 7.4\alpha) = ab.
 % \end{align*} 

\smallskip\noindent{$\bullet$ \bf Case 3:} \textit{Suppose that $J=1$.}  

 From Claims \ref{cl:10A} and  \ref{cl:10B} we get that
 \begin{gather*}
 \mup(\cA)/p^t < 1 + tq + 7.4 =: a,\\
 \mup(\cB)/p^{t+2} < .25/p + 7.4 =: b.
 \end{gather*}   

Again we show that $ab/q^2 < (t+2)^2$.
We have $a/q=8.4/q+t\leq \frac{8.4(t+3)}{t+1}+t$.
On the other hand $b/q$ is decreasing in $p$ for $p<0.15$,
and evaluating it at $p=1/(t+1)$ we have 
$b/q\leq\frac{(t+1) (5 t+153)}{20 t}$.
Using these inequalities we see that $ab/q^2<(t+2)^2$ for $t\geq 16$.

%\blue{For $t \geq 45$ we have that 
% $a_2/a <  (tq + 8.4)/tq < (6/5)$ and  $b_2/b < (.25/p + 7.4)/(1/p) < .25 + 7.4p < (1/2)$ so 
%$a_2b_2 < (3/5)ab < z^2$} from the first case. 
\end{proof}

%%%%%%%%%%%%%%%%%%%%%

\subsection{Extremal cases}%\label{sec:diagonal cases}

Finally, we consider the cases $(s,s')=(0,0)$, $(1,1)$, $(2,2)$. 

\begin{lemma}\label{lem:extremal}
For $(s,s)=(0,0), (1,1), (2,2)$, we have 
\begin{equation}\label{eq:bounded}
\mup(\cA)\mup(\cB)\leq \(\mup(\cF^t_1)\)^2.
\end{equation}
 Moreover,
equality holds if and only if one of the following holds:
\begin{enumerate}\label{eq:equality_2}
\item $\cA=\cB= \cF^t_0$ and $p=\dfrac{1}{t+1}$,
\item $\cA=\cB= \cF^t_1$ and $\dfrac{1}{t+1}\leq p\leq \dfrac{2}{t+3}$,
\item $\cA=\cB= \cF^t_2$ and $p=\dfrac{2}{t+3}$.
\end{enumerate}

%\begin{align}\label{eq:equality_2}
%(i) \hskip 1em & \cA=\cB= \cF^t_0 \hskip 1em \mbox{and} \hskip 4.5em p=\frac{1}{t+1}, \nonumber\\
%(ii)  \hskip 1em &\cA=\cB= \cF^t_1\hskip 1em \mbox{and} \hskip 1em \frac{1}{t+1}\leq p\leq \frac{2}{t+3}, \\
%(iii)   \hskip 1em &\cA=\cB= \cF^t_2 \hskip 1em \mbox{and}\hskip 4.5em  p=\frac{2}{t+3}.\nonumber
%\end{align}

\end{lemma}

\begin{proof}
In these cases, we have $u=v=t$. Recalling~\eqref{eq:D} and~\eqref{eq:range_i}, we let
\begin{align*}
D^t_s(i):=&[t-1]\cup \{t+s, t+s+1,\dots, t+2s\}  \nonumber \\ &\cup \{t+2s+i+2k\in [n] : k=1,2,\dots\}\in \hF^{t}\cap\cF_s^{t}
\end{align*}
for $1\leq i\leq n-t-2s-1=:i_{\max}.$

In order to define $$I=\max\{i \::\: D^t_s(i)\in \cA \} \hskip 1em \mbox{and} \hskip 1em J=\max\{i \::\: D^t_s(i)\in \cB \},$$ the sets $\{i \::\: D^t_s(i)\in \cA \}$ and $\{i \::\: D^t_s(i)\in \cB \}$ should not be empty. Hence, we consider the following two cases separately:
\begin{itemize}
\item \textbf{Case I:} \hskip 1.45em $D^t_s(1)\in \cA$ and $D^t_s(1)\in \cB$.
\item \textbf{Case II:} \hskip 1em $D^t_s(1)\not\in \cA$\hskip 0.7em or \hskip 0.4em $D^t_s(1)\not\in \cB$\end{itemize}

As $D^t_s(1)$ is the shift minimal walk in $\hF^{t}\cap\cF_s^{t}$ for 
$s = 0$ and $1$, and as the subsets $\hA$ and $\hB$ are non-empty, 
we have that Case I holds if $s = 0$ or $1$. So in Case II we may assume that
$s \geq 2$. 

\smallskip\noindent{$\bullet$ \bf Case I:} \textit{Suppose that $D^t_s(1)\in \cA$ and $D^t_s(1)\in \cB$.}

 First, we suppose that $I=J=i_{\max}$.
Since $D^t_s(i_{\max})\in \cA$, Fact~\ref{fact:dual} gives that
\[
\dual_t\(D^t_s(i_{\max})\)=[n]\setminus\{t+s,t+s+1,\dots,t+2s\}
\] 
is not contained in $\cB$. 
Consequently, Fact~\ref{fact:shift_2} gives that each walk $B\in\cB$ satisfies $ B \not\shiftsto \dual_t\(D^t_s(i_{\max})\)$. 
Hence, $\cB\subset\cF_s^t$ holds.
Similarly, $J=i_{\max}$ implies $\cA\subset\cF_s^t$. 
Therefore, we have $$\mup(\cA)\mup(\cB)\leq \(\mup(\cF^t_s)\)^2$$
with equality holding iff
$\cA=\cB=\cF_s^t$. This together with $1/(t+1)\leq p\leq 2/(t+3)$ implies~\eqref{eq:bounded} and~\eqref{eq:equality_2}.

Therefore, we can assume that $I\neq i_{\max}$ or $J\neq i_{\max}$. Without loss of generality, let $I\neq i_{\max}$. The following holds for every $s\geq 0$ (not only for $0\leq s\leq 2$).

\begin{claim}\label{clm:B-F}
If $I\neq i_{\max}$ and $0\leq s< \infty$, then
\begin{equation}\label{eq:F-A} 
\mup(\cF^t_s\setminus \cA)\geq \binom{t+s-1}{s}p^{t+s}q^{s+I+1}(1-\alpha)
\end{equation}
and
\begin{equation}\label{eq:B-F}
\mup(\cB\setminus \cF^t_s)\leq \alpha^{t+I}.
\end{equation}
\end{claim}
\begin{proof}
First, we show~\eqref{eq:F-A}. Consider a walk $W$ that hits $(s,t+s)$ and satisfies $W \shiftsto D^t_s(I+1)$. 
Since $D^t_s(I+1)\not\in \cA$, Fact~\ref{fact:shift_2} gives $W\in\cF^t_s \setminus \cA$.
Also, $W$ must hit $Q_1=(s,t-1)$ and $Q_2=(s+I+1,t+s)$. The number of walks from $(0,0)$ to $Q_1$ is $\binom{t+s-1}{s}$, then there is the unique walk from $Q_1$ to $Q_2$ which hits $(s,t+s)$.
So the weight of the family of all such walks is $\binom{t+s-1}{s}p^{t+s}q^{s+I+1}$.     
To satisfy $W\shiftsto D_{I+1}$, the walk $W$
must not hit the line $y = x + (t - I)$. By Lemma~\ref{lem:mu_F}(i), this happens with probability at least
$1 - \alpha$, which yields \eqref{eq:F-A}.

Next, we show~\eqref{eq:B-F}. Since $D^t_s(I)\in \cA$, we have that $\dual_t\(D^t_s(I)\)\not\in\cB$, and hence, 
each walk $B\in\cB$ must hit $(0,t+s), (1,t+s),\dots,(s,t+s)$,
or $y=x+(t+I)$. 
Note that each walk hitting $(0,t+s), (1,t+s), \dots,\mbox{or } (s,t+s)$ is contained in $\cF^t_s$. Thus, 
each walk $B\in\cB\setminus \cF^t_s$ hits $y=x+(t+I)$. 
Lemma~\ref{lem:mu_F}(i) gives~\eqref{eq:B-F}.
\end{proof}

Claim~\ref{clm:B-F} together with $0\leq s\leq 2$ implies the following.

\begin{corollary}\label{I neq imax}
If $I\neq i_{\max}$ and $0\leq s\leq 2$, then 
$\mup(\cB\setminus\cF^t_s)<0.99\mup(\cF^t_s\setminus\cA)$.
\end{corollary}

\begin{proof}
Inequalities \eqref{eq:F-A} and \eqref{eq:B-F} give that
\begin{align*}%\label{eq:ration}
\frac{\mup(\cF^t_s\setminus \cA)}{\mup(\cB\setminus \cF^t_s)}&\geq \frac{\binom{t+s-1}{s}p^{t+s}q^{s+I+1}(1-\alpha)}
{\alpha^{t+I}}=\binom{t+s-1}{s}p^sq^{t+s+1}\(\frac{q^2}{p}\)^I(1-\alpha) \nonumber \\
&\geq \binom{t+s-1}{s}p^sq^{t+s}\(\frac{q^2}{p}\)(q-p)=\binom{t+s-1}{s}p^{s-1}q^{t+s+2}(q-p),
\end{align*}
where the second inequality holds since $q^2/p>1$ for $p<0.38$. Since $1/(t+1)\leq p\leq 2/(t+3)$, one can easily check that 
$$\binom{t+s-1}{s}p^{s-1}q^{t+s+2}(q-p)>1.02$$ if $s=0$ and $t\geq 17$, or if $s=1$ and $t\geq 12$, or if $s=2$ and $t\geq 22$.
\end{proof}

On the other hand, we also claim that 
\begin{equation}\label{eq:A-F}
\mup(\cA\setminus\cF^t_s)<0.99 \mup(\cF^t_s\setminus\cB).
\end{equation}
 Indeed, if $J\neq i_{\max}$, then Corollary~\ref{I neq imax} gives~\eqref{eq:A-F}.
 Otherwise, $J=i_{\max}$, and hence, $\cA\subset \cF^t_s$. Thus,~\eqref{eq:A-F} trivially holds.

 We infer that
 \begin{align*}
 \mup(\cA)+\mup(\cB)&=\(\mup(\cA\cap \cF^t_s)+\mup(\cA\setminus \cF^t_s)\)+\(\mup(\cB\cap \cF^t_s)+\mup(\cB\setminus \cF^t_s)\) \\
 &<\(\mup(\cA\cap \cF^t_s)+0.99\mup( \cF^t_s\setminus \cA)\)+\(\mup(\cB\cap \cF^t_s)+0.99\mup( \cF^t_s\setminus \cB)\) \\
 &<2\mup(\cF^t_s),
 \end{align*}
where the inequality follows from Corollary~\ref{I neq imax} and~\eqref{eq:A-F}. Therefore,
$$ \sqrt{ \mup(\cA)\mup(\cB)}\leq \frac{ \mup(\cA)+\mup(\cB)}{2}<\mup(\cF^t_s),$$
which gives~\eqref{eq:bounded} without equality.

\smallskip\noindent{$\bullet$ \bf Case II:} \textit{Suppose that $D^t_s(1)\not\in \cA$ or $D^t_s(1)\not\in \cB$.} 

As we observed before Case I, in Case II we may assume that 
$s \geq 2$.  
Also, without loss of generality, we let $D^t_s(1)\not\in \cA$, 
so every $A\in \cA$ satisfies $A\not\shiftsto D^t_s(1)$.

For $1\leq i\leq n-t-5$, let
\begin{equation*}%\label{eq:E_2}
 E(i):=[t-1]\cup\{t+1,t+3,t+4\}\cup \{t+4+i+2j\in[n]:j=1,2,\ldots\}.
\end{equation*} 
For any $A\in \hA\neq \emptyset$, we have $A\shiftsto E(1)$,
 and hence, Fact~\ref{fact:shift_1} gives that $E(1)\in \cA$.
 Since $\{i:E(i)\in\cA\}\neq\emptyset$, the number
$$K:=\max\{i:E(i)\in\cA\}$$
is well-defined.

Let $A\in\hA$. The walk $A$ must hit $(2,t+2)$ without hitting $(0,t)$ or
$(1,t+1)$. Also, since $D^t_s(1)\not\in\cA$, the walk $A$ must hit $(1,t)$. 
The weight of the family of all such walks is $t p^{t+2}q^2$. 
>From $(2,t+2)$, the walk $A$ moves to the right and hits $(3,t+2)$. 
Then it must not hit the line $y=x+t$. Lemma~\ref{lem:epsilon} implies that
this happens with probability less than $q(1-\alpha+\epsilon)$ where $\epsilon\rightarrow 0$ as $n$ tends to $\infty$. Let $n$ be sufficiently large that $\epsilon\leq \alpha$. 
 Then, we have that
\[
 \mup(\hA)< tp^{t+2}q^3(1-\alpha+\epsilon)\leq tp^{t+2}q^3. 
\]
For $\dhA$ and $\tA$ we use the trivial bounds
$\mup(\dhA)\leq \alpha^{t+1}$ and $\mup(\tA)\leq \alpha^{t+1}$ from Lemma~\ref{lem:mu_F}.
Consequently we have
\begin{equation}\label{extremal1a}
 \mup(\cA)=\mup(\hA)+\mup(\dhA)+\mup(\tA)<tp^{t+2}q^3+2\alpha^{t+1}.
\end{equation}

On the other hand,
\begin{align*}
 \dual_t(E(K))&=[t]\cup\{t+2\}\cup\{ t+5, \dots , t+5+K\}\\
 &\qquad\cup\{t+5+2j\in[n]:j=1,2,\ldots\},
\end{align*}
Since $\dual_t(E(K))\not\in\cB$, every walk $B\in\cB$ must hit one of $(0,t+1)$, $(1,t+2)$, $(2,t+2)$,
or the line $y=x+t+K$.
Thus we have
\begin{align}\label{extremal1b}
 \mup(\cB)&\leq p^{t+1} + (t+1)p^{t+2}q +
\((t+1)+\binom{t+2}2\)p^{t+2}q^2 + \alpha^{t+K}\nonumber\\
&\leq\(1+(t+1)pq+\frac{(t+1)(t+4)}2pq^2\)p^{t+1}+\alpha^{t+1}.
\end{align}

Inequalities~\eqref{extremal1a} and~\eqref{extremal1b}
 give that
\begin{align}\label{eq:temp_100}
\frac{\mup(\cA)\mup(\cB)}{\(\mup(\cF^t_1)\)^2} &\leq \frac{\(tp^{t+2}q^3+2\alpha^{t+1}\)\(1+(t+1)pq+\frac{(t+1)(t+4)}{2}pq^2+1/q^{t+1}\)p^{t+1}}{\((t+2)p^{t+1}q\)^2}\nonumber \\
&\leq \frac{\(tpq^3+\frac{2e^2}{q}\)\(1+(t+1)pq+\frac{(t+1)(t+4)}{2}pq^2+\frac{e^2}{q}\)}{(t+2)^2q^2},
\end{align} 
where the second inequality follows from~\eqref{eq:q^t}.
Since $pq, pq^2, pq^3$ and $1/q$ are increasing in $p$ for $p\leq 0.2$, expression~\eqref{eq:temp_100} is maximized when $p=2/(t+3)$. One can check that~\eqref{eq:temp_100} with $p=2/(t+3)$ is at most $0.99$ for $t\geq 180$.
Therefore, for $t\geq 180$,
\begin{equation}\label{eq:t=180}\mup(\cA)\mup(\cB)<0.99 \(\mup(\cF^t_1)\)^2,
\end{equation}
which completes our proof of Lemma~\ref{lem:extremal}.
\end{proof}
We have proved the inequality \eqref{eq:main} under 
Assumption~\ref{assum:maximal_shifted}. The uniqueness of the optimal
families in Theorem~\ref{thm:main} now follows from Lemma~\ref{lem:shifting2}.
This completes the proof of Theorem~\ref{thm:main}.

\appendix

\section{Omitted Calculations}

 \subsection{ Calculations for Claim \ref{clm:g(3,2)}}

 We give here the calculations for the cases $(s,s') = (3,3), (3,2)$ and $(2,2)$, omitted from the proof of Claim \ref{clm:g(3,2)}.   

\paragraph{Case: $(s,s') = (3,3)$}
 Noting that $u = v = t$ we get that 
 \begin{align*}
  f(u,3,p) = f(v,3,p) = \alpha^{t+1} + \binom{t+6}{3}\frac{t+1}{t+4}p^{t+3}q^3(1 - \alpha) < e^2\frac{p^{t+1}}{q} + \frac{(t+6)(t+5)(t+1)}{6}p^{t+3}q^3.
\end{align*}  
 Thus as $\mup(\cF^t_1) > (t+2)p^{t+1}q$ we get that
 \begin{align*}
    \frac{f(u,3,p)}{\mup(\cF^t_1)} & < \frac{e^2}{q^2(t+2)} + \frac{p^2q^2(t+6)(t+5)(t+1)}{6(t+2)} \\
                         & < \frac{2e^2}{t+2} + \frac{4(t+6)(t+5)(t+1)}{6(t+2)(t+3)^2}.
 \end{align*} 
 This is less than $.99$ for $t \geq 52$, its square $\frac{g(3,3)}{(\mup(\cF^t_1))^2}$ is also.  

\paragraph{Case: $(s,s') = (3,2)$}
 Noting that $u = t-1$ and $v = t+1$ we get that 
 \begin{align*}
  f(u,3,p) = \alpha^t + \binom{t+5}{3}\frac{t}{t+3}p^{t+2}q^3(1 - \alpha) < e^2p^t + \frac{(t+5)(t+4)t}{6}p^{t+2}q^3,
\end{align*}  
% where we have used that $e^{-2} \leq q^t \leq 1/2$, 
and we get 
 \begin{align*}
  f(v,2,p) = \alpha^{t+2} + \binom{t+5}{2}\frac{t+2}{t+4}p^{t+3}q^2(1 - \alpha) < \frac{e^2p^{t+2}}{q^2} + \frac{(t+5)(t+2)}{2}p^{t+3}q^2.
\end{align*}  
 Thus as $(\mup(\cF^t_1))^2 > (t+2)^2p^{2t+2}q^2$ we get that
 \begin{align*}
    \frac{g(3,2)}{(\mup(\cF^t_1))^2} & 
< \frac{e^4}{q^4(t+2)^2} + \frac{e^2p(t+5)}{2(t+2)} 
+ \frac{e^2p^2(t+5)(t+4)t}{6q(t+2)^2} 
+ \frac{p^3q^3(t+5)^2(t+4)t}{12(t+2)}\\
  & < \frac{2e^4}{(t+2)^2} + \frac{e^2(t+5)}{(t+2)(t+3)} 
+ \frac{4e^2(t+5)(t+4)t}{6(t+2)^{2}(t+3)(t+1)} + \frac{2(t+5)^2(t+4)t}{3(t+2)(t+3)^3}.
 \end{align*} 
 This is less than $.99$ for $t \geq 51$. 
 
\paragraph{Case: $(s,s') = (2,0)$}
 Noting that $u = t-2$ and $v = t+2$ we get that 
 \begin{align*}
  f(u,2,p) = \alpha^{t-1} + \binom{t+2}{2}\frac{t-1}{t+1}p^{t}q^2(1 - \alpha) < e^2\frac{p^{t-1}}q + \frac{(t+2)(t-1)}{2}p^{t}q^2,
\end{align*}  
% where we have used that $e^{-2} \leq q^t \leq 1/2$, 
and we get 
 \begin{align*}
f(v,0,p) = \alpha^{t+3} + p^{t+2}q(1 - \alpha) 
< \frac{e^2p^{t+3}}{q^3} + p^{t+2}q.
\end{align*}  
 Thus as $(\mup(\cF^t_1))^2 > (t+2)^2p^{2t+2}q^2$ we get that
 \begin{align*}
    \frac{g(2,0)}{(\mup(\cF^t_1))^2} 
        & < \frac{e^4}{q^6(t+2)^2} 
          + \frac{e^2}{(t+2)^2pq^2}
          + \frac{e^2p(t-1)}{2q^3(t+2)} 
          + \frac{(t-1)q}{2(t+2)}\\
        & < \frac{2e^4}{(t+2)^2} 
          + \frac{e^2(t+1)^3}{(t+2)^2t^2}
          + \frac{2e^2(t-1)}{(t+3)(t+2)} 
          + \frac{(t-1)(t+1)}{2(t+2)(t+3)}.
 \end{align*} 
 This is less than $.99$ for $t \geq 42$. 

\subsection{Proofs of Claims \ref{cl:23B}, \ref{cl:10A}
  and \ref{cl:10B} }

%\begin{claim}
% For all $p$ and $t \geq 19$ there exists $\epsilon$ such that   
%  \[ \mup(\cB) / p^{t+2} \leq  b_1(p,t)+b_2(p,t)\alpha^{I-1}- b_3(p,t)q^{J-2} + \epsilon, \] 
%where %$b_1(p,t):=1 + (t+2)q$, $ b_2(p,t):=q^{-(t+2)} -1 - (t+2)p < 4.5$, 
%%and $b_3(p,t):=(t+1)q^3(1-\alpha)$.
%\begin{align*}
% b_1(p,t)&:=1 + (t+2)q, \\
% b_2(p,t)&:=q^{-(t+2)} -1 - (t+2)p < 4.5, \\
% b_3(p,t)&:=(t+1)q^4(1-\alpha) > (t - 7) q 
%\end{align*}
%\end{claim} 

\begin{proof}[Proof of Claim~\ref{cl:23B}]\claimproof
Let $\epsilon>0$ be given and let $\delta=\epsilon/b_1(p,t)$.
As $s' = 1$ we have $\mu_p(\cB)=\mu_p(\tB)+\mu_p(\cB_1)$.  
Noting that $D^{t-1}_2(I)\in\cA$ and 
\[
\dual_t(D^{t-1}_2(I))=[t+1]\cup[t+4,t+4+I]\cup\{t+6+I,t+8+I,\ldots\}\not\in\cB 
\]
we see that every walk in $\cB$ must hit at least one of 
$(0,t+2)$, $(1,t+2)$, and the line $L:y=x+(t+1+I)$. 
Also all walks in $\tB\subset\cF^v=\cF^{t+1}$ hit the line $L':y=x+(t+2)$.
Thus we get
\begin{align}\label{eq:tB}
\mu_p(\tB)&\leq \mu_p(\mbox{walks in $\tB$ hitting $L$})\nonumber \\ 
&\quad+\mu_p(\mbox{walks in $\tB$ not hitting $L$ but hitting $(0,t+2)$})\nonumber \\
&\quad+\mu_p(\mbox{walks in $\tB$ not hitting $L$ or $(0,t+2)$ 
but hitting $(1,t+2)$ and $L'$})\nonumber \\
&\leq \alpha^{t+1+I}+p^{t+2}(1-\alpha^{I-1}+\delta) 
+(t+2)p^{t+2}q(\alpha-\alpha^{I}+\delta/2). 
\end{align} 

On the other hand, as $D^{t+1}_1(J+1)\not\in\cB$, we have that 
$\cB_1 \subset \bF^{t+1}_1 \setminus \cW$ where 
\[ \cW = \{W\in\hF^{t+1}\cap\cF_1^{t+1}:W\to D_1^{t+1}(J+1)\}. \]
Walks in $\bF^{t+1}_1$ hit $(1,t+2)$ without hitting $(0,t+2)$
%\blue{[walks in $\bF^{t+1}_1$ can hit $(0,t+1)$ on $y = x + t + 1$]}
then do not go above $y = x + t + 1$ so we have that
\begin{equation}\label{eq:hB}
 \mu_p(\bF^{t+1}_1) \leq  (t+2)p^{t+2}q(1-\alpha+\delta/2).
\end{equation}
Walks in $\cW$ are those in $\bF^{t+1}_1$ that after hitting $(1,t+2)$ go over
to $(J + 2, t+2)$, which is
on the line $y = x + t - J$, and then 
never cross this line.   
So 
\begin{equation}\label{eq:eB}
 \mu_p(\cW) \geq (t+1)p^{t+2}q^{J+2}(1-\alpha).
\end{equation}

Combining~\eqref{eq:tB}, \eqref{eq:hB}, and \eqref{eq:eB} yields that
\[
\mu_p(\cB)/p^{t+2} \leq q^{-(t+2)}\alpha^{I-1} + (1-\alpha^{I-1}+\delta) 
+ (t+2)q(1 - \alpha^I + \delta) - (t+1)q^{J+2}(1 - \alpha).
\]
%where $\epsilon=\max\{\epsilon_1, \epsilon_2 + \epsilon_3\}$.
Rearranging we get 
\[
\mu_p(\cB) / p^{t+2} \leq (1+\delta)b_1(p,t)+b_2(p,t)\alpha^{I-1} - b_3(p,t)q^{J-2}, 
 \]
as needed.
%where $b_1(p,t)= 1 + (t+2)q$ and $b_2(p,t) =q^{-(t+2)} -1 - (t+2)p$, and 
%$b_3 = (t+1)q^4(1-\alpha)$ as in the statement of the claim. 
%As $\epsilon$ is a function of $p$ and $t$, this is equivalent to the claim.

Since $\frac{\partial}{\partial p} b_2(p,t)=(t+2)(q^{-(t+3)}-1)>0$ 
it follows that
$b_2(p,t)\leq b_2(\frac2{t+3},t)=(\frac{t+3}{t+1})^{t+2} - 1 -\frac{2(t+2)}{t+3}$,
which is decreasing in $t$, so for $t\geq 18$, 
$b_2(p,t) \leq b_2(\frac2{t+3},t) \leq b_2(\frac2{21},18)< 4.5$.

As $q^3(1 - \alpha)$ is decreasing in $p$, 
we get, by evaluating it at $p = 2/(t+3)$, that 
$  b_3(p,t)   \geq q\frac{(t-1)(t+1)^3}{(t+3)^3} > (t-7)q$.
%\begin{align*}
%  b_3(p,t) &  \red{\geq \frac{(t-1)(t+1)^4}{(t+3)^4}} 
%    %  & = q \( t -  \frac{7t^3 + 27t^2 + 29t + 1}{t^3 + 9t^2 + 27t + 27} \)\\
%       > \red{t - 9.}
%\end{align*}
\end{proof}    

%\setcounter{theorem}{24}
% \begin{claim}[Claim \ref{cl:10A}]
%   For all $p$ and $t$ there exists $\epsilon$ such that
%    \[ \mu_p(\cA)/p^t  < a_1(p,t)+a_2(p,t)\alpha^{J-1} - a_3(p,t)q^{I-1} + \epsilon \]
%  where 
%\begin{align*}
% a_1(p,t):=1 + tq, \qquad 
% a_2(p,t):=q^{-t} < 7.4, \qquad 
% a_3(p,t):=(t-1)q^3(1-\alpha) > (t-7)q  
%% a_3(p,t)&:=\frac{(t+2)(t-1)}2pq^4(1-\alpha).  
%\end{align*}
% \end{claim}

 \begin{proof}[Proof of Claim~\ref{cl:10A}]\claimproof
Let $\epsilon$ be given and let $\delta=\epsilon/(tq)$.
We use that $\cA = \tA \cup \cA_1$.
%$\cA = \tA \cup \cA_1 \subset \tA \cup \subset \tA \cup 
%(\bF_1^{t-1} \red{\setminus} \cW)$ where
%     \[ \cW = \{W\in\hF^{t-1}\cap\cF_1^{t-1}:W\to D^{t-1}_1(I+1)\}. \]
   We have by arguments similar to in the proof of Claim \ref{cl:22A}, or from 
   the inequalities (11) and (12) of \cite{F},  that
%      \[ \mup(\tA) +  \mup(\bF_1^{t-1}) < \alpha^{t+J-1} + p^t + tp^tq + \epsilon' \]

\begin{align*}
\mup(\tA)  &\leq \alpha^{t+J-1} + p^t + tp^tq\alpha,\\
\mup(\cA_1)& \leq\mup(\bF_1^{t-1}) < tp^tq(1-\alpha+\delta),
\end{align*}

where the second inequality follows by choosing $n$ sufficiently large.
  
We also use that $\cA_1 \subset \bF_1^{t-1} \setminus \cW$, where
     \[ \cW = \{W\in\hF^{t-1}\cap\cF_1^{t-1}:W\to D^{t-1}_1(I+1)\}. \]
   A path $W \in \cW$ hits $(1,t)$ without hitting $(0,t)$ and then goes 
   over to $(I+2,t)$ on the line $y = x + (t - I - 2)$, and afterwards never
   crosses this line.
   So 
    \[ \mup(\cW) > (t-1)p^tq^{I+2}(1-\alpha). \]
   Together, this gives 

\begin{align*}
\mup(\cA) &< \alpha^{t+J-1} + p^t + tp^tq (1+\delta)-(t-1)p^tq^{I+2}(1-\alpha)\\
&<p^t\left(
(1+tq) + \frac{\alpha^{J-1}}{q^t}-(t-1)q^{I+2}(1-\alpha)+tq\delta
\right),
\end{align*}

which yields the main inequality of the claim. 
      
    That $q^{-t} < e^2 < 7.4$ was observed in the proof of Claim \ref{cl:22A}
   and that $(t-1)q^3(1-\alpha) > (t-7)q$
   can be shown as in the proof of Claim \ref{cl:23B}.  
 \end{proof}

%\begin{claim}
% For all $p$ and $t \geq 20$ there exists $\epsilon$ such that   
%  \[ \mup(\cB) / p^{t+2} \leq  b_1(p,t)+b_2(p,t)\alpha^{I-1}- b_3(p,t)q^{J-1} + \epsilon, \] 
%where %$b_1(p,t):=1 + (t+2)q$, $ b_2(p,t):=q^{-(t+2)} -1 - (t+2)p < 4.5$, 
%%and $b_3(p,t):=(t+1)q^3(1-\alpha)$.
%\begin{align*}
% b_1(p,t):=1/p, \qquad
% b_2(p,t):=q^{-(t+2)} < 7.4, \qquad
% b_3(p,t):=(q^2/p)(1-\alpha) > .75/p  
%\end{align*}
%\end{claim} 

\begin{proof}[Proof of Claim~\ref{cl:10B}]\claimproof
Let $\epsilon$ be given and let $\delta=p\epsilon$.
 We use that $\cB = \tB \cup \cB_0 \subset \tB \cup 
(\bF^{t+1}_0 \setminus \cW),$ where  
    \[ \cW = \{W\in\hF^{t+1}\cap\cF_0^{t+1}:W\to D_0^{t+1}(J+1)\}. \]
 From the inequalities (14) and (15) of \cite{F} we have that
  \[ \mup(\tB) + \mup(\bF_0^{t+1}) \leq \alpha^{t + 1 + I} + p^{t+1}(1+\delta). \]
 A path $W$ in $\cW$ hits $(0,t+1)$ goes over to $(J+1,t+1)$ on the line
 $y = x + (t-J)$, and then never crosses this line, so 
     \[ \mup(\cW) \geq p^{t+1}q^{J+1}(1-\alpha). \]
Consequently it follows that
   \[ \mup(\cB)  \leq p^{t+1} + \alpha^{t + 1 + I} -  p^{t+1}q^{J+1}(1-\alpha)
     + p^{t+1}\delta,\]
   which yields the main inequality of the claim. 
  The bound for $b_2(p,t)$ was shown in the proof of the previous claim, and the bound
  for $b_3(p,t)$ can be verified for $t \geq 20$. 
\end{proof}

\subsection{Other Computations}\label{app:othercomp}

 We verify that $y/z$ is decreasing for Case 1 of Claim \ref{cl:23B}.
% by showing that $zy' - yz' > 0$. 
 Recall that $z = t+2 - (t+1)p = (t+2)q + p$ and that
 $y = 1 + (t+2)q + 4.5\alpha^2 = z + (4.5\alpha^2 +q)$,
 so $y/z = 1 + \frac{r}{z}$, where $r =4.5\alpha^2 +q$.
 We show this is increasing by showing that 
 $zr' - rz' > 0$, where $r'$ and $z'$ denote derivatives with respect to $p$.  
 Noting that $\frac{\partial}{\partial p}q = -1$ and
  $\frac{\partial}{\partial p}\alpha = q^{-2}$
 we compute
 \begin{align*}
   zr' - rz' &=  
z(9\alpha q^{-2}-1)-r(-(t+1))\\
%\((t+2)q + p\)\(9\frac{p}{q^3} - 1\) -      \(1 +  4.5\alpha^2-p\)\(-(t+1)\) \\
%       &= (t+1) + 4.5(t+1)\alpha^2 - p(t+1) + 9(t+2)\frac{p}{q^2} + 9\frac{p^2}{q^3}       -p - (t+2)q \\
       &=  4.5(t+1)\alpha^2  + 9(t+2)\frac{p}{q^2} + 9\frac{p^2}{q^3} - 1 \\ 
       &>   9(t+2)\frac{p}{q^2} - 1 > 9(t+2)\frac{t+1}{t^2} - 1 > 8, 
 \end{align*} 
  as needed.  
  In the last line we use $p = \frac{1}{t+1}$ as $p/q^2$ is increasing in $p$.     

\begin{thebibliography}{99}

\bibitem{AK1}
R.~Ahlswede, L.H.~Khachatrian.
\newblock The complete intersection theorem for systems of finite sets.
\newblock {\em European J.\ Combin.}, 18:125--136, 1997.

\bibitem{AK2}
R.~Ahlswede, L.H.~Khachatrian.
\newblock A Pushing-pulling method: new proofs of intersection theorems.
\newblock {\em Combinatorica}, 19:1--15, 1999.



\bibitem{AK-p}
R.~Ahlswede, L.H.~Khachatrian.
\newblock The diametric theorem in Hamming spaces--optimal anticodes.
\newblock {\em Adv.\ in Appl.\ Math.}, 20:429--449, 1998.

%\bibitem{BE}
%C.~{Bey}, K.~Engel.
%\newblock Old and new results for the weighted $t$-intersection
%problem via AK-methods.
%\newblock {\em Numbers, Information and Complexity, Althofer, Ingo,
%Eds. et al., Dordrecht}, Kluwer Academic Publishers, 45--74, 2000.

%\bibitem{B}
%P.~Borg.
%\newblock The maximum product of sizes of cross-$t$-intersecting uniform families.
%\newblock {\em Australasian Journal of Combinatorics}, 60(1):69--78, 2014.


%\bibitem{DS}
%I.~Dinur, S.~Safra.
%\newblock On the Hardness of Approximating Minimum Vertex-Cover.
%\newblock {\em Annals of Mathematics}, 162:439-485, 2005.

\bibitem{EKR}
P.~{Erd\H os}, C.~Ko, R.~Rado.
\newblock Intersection theorems for systems of finite sets.
\newblock {\em Quart.\ J.\ Math.\ Oxford (2)}, 12:313--320, 1961.

\bibitem{Fckt}
P.~Frankl.
\newblock The Erd\H{o}s--Ko--Rado theorem is true for $n=ckt$.
\newblock {\em Combinatorics (Proc. Fifth Hungarian Colloq., Keszthey,
	1976), Vol.~I}, 365--375,
Colloq.\ math.\ Soc.\ J{\' a}nos Bolyai, 18, North--Holland, 1978.

%\bibitem{FFintseq}
%P.~Frankl, Z.~F\"uredi.
%\newblock The Erd{\H o}s--Ko--Rado theorem for integer sequences.
%\newblock {\em SIAM J. Alg. Disc. Math.}, 1:376--381, 1980.

\bibitem{FF}
P.~Frankl,  Z.~F\"uredi.
\newblock Beyond the Erd\H os--Ko--Rado theorem.
\newblock {\em J.\ Combin.\ Theory (A)}, 56 (1991) 182--194.

\bibitem{F}
P.~Frankl, S.~J.~Lee, M.~Siggers, N.~Tokushige,
\newblock An Erd\H{o}s--Ko--Rado theorem for cross $t$-intersecting families.
\newblock {\em J. Combin. Theory (A) }, 128:207--249, 2014.

%\bibitem{FT}
%P.~Frankl, N.~Tokushige.
%\newblock The Erd{\H o}s--Ko--Rado theorem for integer sequences.
%\newblock{\em Combinatorica}, 19 (1999) 55--63.

%\bibitem{FT2002}
%P.~Frankl, N.~Tokushige.
%Weighted $3$-wise $2$-intersecting families.
%{\em J.{} Comb.{} Theory (A)}, Vol.~100 (2002) 94--115.

%\bibitem{FTw}
%P.~Frankl, N.~Tokushige.
%Weighted multiply intersecting families.
%{\em Studia Sci.\ Math.\ Hungarica}, Vol.~40 (2003) 287--291.

%\bibitem{Fri}
%E.~Friedgut.
%\newblock On the measure of intersecting families, uniqueness and stability.
%{\em Combinatorica} 28 (2008) 503--528.

%\bibitem{G}
%M.~Gromov.
%\newblock Singularities, expanders and topology of maps.
%Part 2: From combinatorics to topology via algebraic isoperimetry.
%\newblock {\em Geom.\ Funct.\ Anal.} 20 (2010) 416--526.

%\bibitem{Ka}
%G. Katona.
%\newblock A simple proof of the Erd\"{o}s--Chao Ko--Rado theorem.
%\newblock{\em  J.{} Combin.{} Theory (B)}, 13 (1972) 183--184.

%\bibitem{KS}
%G.~Kindler, S.~Safra.
%\newblock Noise-resistant boolean functions are juntas.
%preprint.

%\bibitem{K}
%D.~J.~Kleitman.
%\newblock Families of non-disjoint subsets.
%\newblock {\em J.{} Combin.{} Theory}, 1 (1966) 153--155.


%\bibitem{MT}
%M.~Matsumoto, N.~Tokushige.
%\newblock The exact bound in the Erd{\H o}s--Ko--Rado theorem for
%cross-intersecting families.
%\newblock{\em  J.{} Combin.{} Theory (A)}, 52 (1989) 90--97.

%\bibitem{M}
%A.~Moon.
%\newblock An analogue of the Erd\"os--Ko--Rado Theorem for the
%Hamming Schemes $H(n,q)$.
%\newblock{\em  J.{} of Combin.{} Theory (A)}, 32 (1982) 386--390.

\bibitem{PT}
J. Pach, G. Tardos.
\newblock Cross-intersecting families of vectors.
\newblock {\em  Graphs and Combinatorics} 31 (2015) (2), 477–495.


%\bibitem{P}
%L.~Pyber.
%\newblock A new generalization of the Erd\H os--Ko--Rado theorem.
%\newblock{\em J. Combin. Theory (A)} 43:85--90, 1986.

%\bibitem{ST}
%S.~Suda, H.~Tanaka.
%\newblock A cross-intersection theorem for vector spaces based on semidefinite
%programming.
%\newblock{\em Bull. Lond. Math. Soc.}, to appear.

%\bibitem{Tuvsw}
%N.~Tokushige.
%\newblock Intersecting families --- uniform versus weighted.
%\newblock{\em Ryukyu Math.\ J.}, 18 (2005) 89--103.

%\bibitem{Tbalaton}
%N.~Tokushige.
%The random walk method for intersecting families.
%\newblock
%in {\em Horizons of combinatorics, Bolyai society mathematical studies},
%17 (2008) 215--224.

%\bibitem{T0}
%N.~Tokushige.
%\newblock On cross $t$-intersecting families of sets.
%\newblock {\em J.\ Combin.\ Theory (A)}, 117 (2010) 1167--1177.

%\bibitem{Teigen1}
%N.~Tokushige.
%\newblock The eigenvalue method for cross $t$-intersecting families.
%\newblock {\em Journal of Algebraic Combinatorics},
%38 (2013) 653--662.

%\bibitem{Teigen2}
%N.~Tokushige.
%\newblock Cross $t$-intersecting integer sequences from weighted
%Erd\H os-Ko-Rado.
%\newblock {\em Combinatorics, Probability and Computing}, 22 (2013) 622-637.

%\bibitem{Weichsel}
%P.~M.~Weichsel.
%\newblock The Kronecker product of graphs.
%\newblock {\em Proc.\ Amer.\ Math.\ Soc.} 13 (1962) 47--52.

\bibitem{W}
R.M.~Wilson.
\newblock The exact bound in the Erd\H{o}s--Ko--Rado theorem.
\newblock {\em Combinatorica}, 4 (1984) 247--257.


\end{thebibliography}
\end{document}